\documentclass[12pt]{amsart}

\usepackage{comment}
\usepackage[rgb,dvipsnames]{xcolor}
\usepackage{stmaryrd}
\usepackage{amsfonts}
\usepackage{amssymb}
\usepackage{amsmath,amsthm}
\usepackage{latexsym}
\usepackage{amscd}
\usepackage{esint}
\usepackage{mathrsfs}
\usepackage{dsfont}
\usepackage{setspace}
\usepackage{enumitem}
\usepackage[margin=1.25in]{geometry}
\usepackage{bm}

\usepackage{hyperref}
\usepackage{comment}
\usepackage{enumitem}

\newtheorem{theorem}{Theorem}[section]
\newtheorem{corollary}[theorem]{Corollary}
\newtheorem{claim}[]{Claim}
\newtheorem{lemma}[theorem]{Lemma}

\newtheorem{proposition}[theorem]{Proposition}

\theoremstyle{definition}
\newtheorem{definition}[theorem]{Definition}

\newtheorem{question}[theorem]{Question}

\theoremstyle{remark}
\newtheorem{remark}[theorem]{Remark}

\numberwithin{equation}{section}

\newcommand{\an}{\textnormal{An}^G}

\newcommand{\V}{\mathcal{V}}

\newcommand{\IV}{\mathcal{IV}}
\newcommand{\R}{\mathbb{R}}
\newcommand{\N}{\mathbb{N}}
\newcommand{\mH}{\mathcal{H}}
\newcommand{\F}{\mathcal{F}}

\newcommand{\B}{\mathcal{B}}
\newcommand{\C}{\mathcal{C}}

\newcommand{\mZ}{\mathbb{Z}}
\newcommand{\Z}{\mathcal{Z}}
\newcommand{\mS}{\mathcal{S}}

\newcommand{\M}{\mathbf{M}}

\newcommand{\mf}{\mathbf{f}}

\newcommand{\mF}{\mathbf{F}}
\newcommand{\mI}{\mathbf{I}}

\newcommand{\spt}{\operatorname{spt}}
\newcommand{\dist}{\operatorname{dist}}
\newcommand{\Div}{\operatorname{div}}

\newcommand{\inj}{\operatorname{inj}}

\newcommand{\interior}{\operatorname{int}}
\newcommand{\Ric}{\operatorname{Ric}}
\newcommand{\Clos}{\operatorname{Clos}}

\newcommand{\rom}[1]{\expandafter\romannumeral #1}

\title{Equivariant Morse index of min-max $G$-invariant minimal hypersurfaces}
\author{Tongrui Wang}
\address{Institute for Theoretical Sciences, Westlake Institute for Advanced Study (Westlake University), Hangzhou, Zhejiang, 310024, China}
\email{wangtongrui@westlake.edu.cn}

\begin{document}
\maketitle
\begin{abstract}
For a closed Riemannian manifold $M^{n+1}$ with a compact Lie group $G$ acting as isometries, the equivariant min-max theory gives the existence and the potential abundance of minimal $G$-invariant hypersurfaces provided $3\leq {\rm codim}(G\cdot p) \leq 7$ for all $p\in M$. 
In this paper, we show a compactness theorem for these min-max minimal $G$-hypersurfaces and construct a $G$-invariant Jacobi field on the limit. 
Combining with an equivariant bumpy metrics theorem, we obtain a $C^\infty_G$-generic finiteness result for min-max $G$-hypersurfaces with area uniformly bounded. 
As a main application, we further generalize the Morse index estimates for min-max minimal hypersurfaces to the equivariant setting. 
Namely, the closed $G$-invariant minimal hypersurface $\Sigma\subset M$ constructed by the equivariant min-max on a $k$-dimensional homotopy class can be chosen to satisfy ${\rm Index}_G(\Sigma)\leq k$. 
\end{abstract}

\section{Introduction}

In the 1960s, Almgren \cite{almgren1962homotopy}\cite{almgren1965theory} initiated the min-max theory for the area functional to construct minimal submanifolds in any dimension within every compact Riemannian manifold. 
About 20 years later, Pitts \cite{pitts2014existence} and Schoen-Simon \cite{schoen1981regularity} established the regularity theory for codimension one hypersurfaces in closed manifolds, leading to the famous Almgren-Pitts Min-max Theory. 
Recently, this theory has undergone significant development and is now employed extensively in the study of minimal hypersurfaces (\cite{marques2017existence}\cite{song2018existence}) and related geometric variational problems. 

In the study of min-max theory, determining the Morse index of the minimal hypersurface constructed by the min-max method is a crucial problem. 
In the works of Marques-Neves \cite{marques2012rigidity} and Zhou \cite{zhou2015min}\cite{zhou2017min}, a Morse index upper bound was investigated for the one-parameter situation. 
Later in \cite{marques2016morse}, Marques-Neves presented the first general estimates of the Morse index for min-max minimal hypersurfaces. 
Then, as one of the most important applications of min-max theory, the Morse-theoretic description of the minimal hypersurfaces set was demonstrated through a series of studies (\cite{marques2021morse}\cite{zhou2020multiplicity}). 
We refer to \cite{chodosh2020minimal} for a variant in the Allen-Cahn setting.

An extension of min-max theory is to consider the area variational problem under certain symmetric restrictions, which is known as the {\em equivariant min-max theory}. 
Pitts-Rubinstein \cite{pitts1987applications}\cite{pitts1988equivariant} made the initial claim that the equivariant min-max method can be used to construct an equivariant minimal surface with control on its index and genus in any three-dimensional closed Riemannian manifold under the isometric action of a finite group. 
This result was first achieved by Ketover \cite{ketover2016equivariant}\cite{ketover2019genus} using the min-max theory under smooth settings (see \cite{smith1983existence}\cite{de2013existence}). 

More generally, we can consider a closed Riemannian manifold $(M^{n+1}, g_{_M})$ with a compact Lie group $G$ acting isometrically on $M$. 
Then, a necessary condition for the existence of minimal $G$-hypersurfaces is ${\rm Cohom}(G):=\min_{p\in M}{\rm codim} (G\cdot p)\geq 1$, i.e. $\dim(M/G)\geq 1$. 
Otherwise, $G$ acts transitively, and there is no $G$-hypersurface. 
For the case that ${\rm Cohom}(G)= 1$, the minimal $G$-hypersurface can be obtained by taking the orbit with the largest area (cf. \cite[Section 7]{liu2021existence}). 
If ${\rm Cohom}(G)=2$, regularity problems arise in the equivariant min-max as in the min-max for closed geodesic curves under Almgren-Pitts' setting. 
Namely, if one performs Almgren-Pitts' min-max on the $2$-dimensional smooth part of $M/G$, then one would even get interior singularities of the min-max limit. 
Hence, the equivariant min-max theory mainly considers the case that ${\rm Cohom}(G)\geq 3$, which is similar to the dimension assumption $n+1\geq 3$ in Almgren-Pitts' theory. 

For compact connected $G$ with ${\rm Cohom}(G)\geq 3$, Liu \cite{liu2021existence} demonstrated an equivariant min-max construction in the smooth setting. 
Moreover, in the setting of Almgren-Pitts, the author \cite{wang2022min}\cite{wang2023min} has also adapted the multi-parameter min-max theory to an equivariant version (without the connectivity assumption of $G$), which further suggests the abundance of $G$-invariant minimal hypersurfaces provided $\Ric_M >0$. 
It now seems plausible to search for the Morse index estimates of the min-max $G$-invariant minimal hypersurfaces. 

Roughly speaking, the multi-parameter equivariant min-max theory is to find the saddle points for the mass norm among a homotopy class of some $k$-complex ($k\in \mZ^+$) in the $G$-invariant $n$-cycles space $\Z_n^G(M;\mZ_2)$. Since the mass supremum is only taken over a $k$-complex, the saddle point shall have $k$ directions in $\Z_n^G(M;\mZ_2)$ with mass non-increasing. 
Therefore, the {\em equivariant Morse index} ${\rm Index}_G(\Sigma)$ of such min-max minimal $G$-hypersurface $\Sigma$ shall be bounded by $k$, where ${\rm Index}_G(\Sigma)$ is the number of negative eigenvalues (counted with multiplicities) of the Jacobi operator $L_\Sigma$ restricted on the space of $G$-invariant normal vector fields (Definition \ref{Def: G-index G-stable}).

In this paper, our main result confirms the above heuristics by establishing upper equivariant Morse index bounds for min-max $G$-invariant minimal hypersurfaces, which is an equivariant generalization of \cite[Theorem 1.2]{marques2016morse}. 
Specifically, we have the following simplified main theorem 
(see Theorem \ref{Thm: main theorem} for a precise statement):
\begin{theorem}\label{Thm: 1}
	Let $(M^{n+1}, g_{_M})$ be a closed Riemannian manifold and $G$ be a compact Lie group acting as isometries on $M$ so that $3\leq {\rm codim}(G\cdot p) \leq 7$ for all $p\in M$. 
	Then the closed smooth embedded $G$-invariant minimal hypersurface $\Sigma\subset M$ produced by the equivariant min-max over a $k$-dimensional homotopy class can be chosen to satisfy ${\rm Index}_G(\Sigma)\leq k$. 
	%
\end{theorem}

Although we only consider closed manifolds in this paper, we conjecture the same conclusion under the {\em free boundary equivariant setting} (c.f. \cite{wang2023min}). 
Indeed, for the case that $n=2$, $\partial M^{n+1}\neq\emptyset$ is strictly mean convex, and $G$ is a finite group acting as orientation-preserving isometries, Franz \cite{franz2021equivariant} showed a $G$-index estimate for free boundary minimal $G$-surfaces obtained via an equivariant min-max procedure under the smooth setting of Ketover \cite{ketover2019genus}. 
The strictly mean convex assumption on $\partial M$ is used to guarantee the free boundary minimal hypersurfaces (briefly FBMHs) are all {\em properly embedded} (\cite[Remark 1.1]{franz2021equivariant}), i.e. $\partial \Sigma = \Sigma\cap \partial M$. 
Nevertheless, the recent studies of improperly embedded FBMHs (\cite{guang2021min}\cite{guang2021compactness}\cite{wang2019compactness}) may offer promising tools to remove this assumption.

Potential applications of our results include the generic denseness for minimal $G$-hypersurfaces and the Morse theory for the $G$-hypersurfaces space. 
Furthermore, a more distant prize could be the equivariant Morse-Bott theory for the area functional.

The proof of Theorem \ref{Thm: 1} is mainly based on an equivariant version of the deformation procedure developed by Marques-Neves in \cite{marques2016morse}. 
The key novelties include an equivariant bumpy metrics theorem \ref{Thm: bumpy metric} and a compactness theorem \ref{Thm: 3} for min-max $G$-hypersurfaces, which generalize \cite{white2017bumpy} and \cite{sharp2017compactness} to the equivariant case. 
We will cover both of these aspects more in the following subsections. 

Additionally, it should be noted that the equivariant min-max theory employed in this paper is an improved version generalizing from the author's previous work \cite{wang2022min}\cite{wang2023min}. 
Compared with the constructions in \cite{wang2022min}, we here do {\em not} need to assume that the union of non-principal orbits $M\setminus M^{prin}$ is a submanifold with $\dim(M\setminus M^{prin})\leq n-2$, and the dimension assumptions $3\leq n+1\leq 7$ and ${\rm Cohom}(G)\geq 3$ are now generalized by $3\leq {\rm codim}(G\cdot p) \leq 7$ for all $p\in M$. 

Specifically, in the equivariant Almgren-Pitts min-max construction, one technical challenge is that the proof of \cite[Theorem 3.9]{pitts2014existence} is not readily applicable in the equivariant scenario, since the injectivity radii of orbits generally do not admit a uniform positive lower bound. 
Using an additional assumption that $M\setminus M^{prin}$ is a submanifold with $\dim(M\setminus M^{prin})\leq n-2$, the author first addressed this issue in \cite{wang2022min} by considering $M\setminus M^{prin}$ as a whole so that $\inf_{p\in {\rm An}(M\setminus M^{prin},s,t)}{\rm Inj}(G\cdot p) >0$ and the singularity at $M\setminus M^{prin}$ of a stable minimal hypersurface can be removed by \cite{wickramasekera2014general}. 
Motivated by Li-Zhou's work \cite[Appendix B]{li2021min}, this difficulty was eventually resolved in \cite{wang2023min} by treating each orbit type stratum as a `boundary', eliminating the need for the extra assumption on $M\setminus M^{prin}$. 

Furthermore, since the singular set of a $G$-hypersurface is also $G$-invariant, the optimal regularity for stable minimal hypersurfaces \cite{schoen1981regularity}\cite{wickramasekera2014general} indicates that the assumption $3\leq {\rm codim}(G\cdot p) \leq 7$ is adequate for our regularity theory.

\subsection{Equivariant bumpy metrics theorem}

Based on the structure theorem \cite[Theorem 2.1]{white1991space}, White showed in \cite[Theorem 2.2]{white1991space} that for a generic $C^k$ Riemannian metric on $M^{n+1}$, every closed embedded minimal submanifold in $M$ is non-degenerate. 
Roughly speaking, this suggests that for a generic Riemannian metric, the area functional is a `Morse function' on the space of closed embedded hypersurfaces. 
This result has a wide range of applications in the investigation of the generic features of minimal hypersurfaces, e.g. the generic denseness (\cite{irie2018density}\cite{marques2019equidistribution}), generic regularity (\cite{chodosh2020singular}\cite{li2020generic}), and Morse theory (\cite{marques2020morse}\cite{marques2021morse}\cite{zhou2020multiplicity}) for minimal hypersurfaces. 
We also refer to \cite{ambrozio2018compactness} for a free boundary version of the bumpy metrics theorem. 

Suppose $G$ is a compact Lie group acting on $M^{n+1}$ by diffeomorphisms with $\dim(G\cdot p)\leq n-1$ for all $p\in M$. 
For any $G$-invariant Riemannian metric $g_{_M}$, the symmetry leads to the degeneracy of some non-equivariant minimal hypersurfaces. 
To be exact, let $g(t)\subset G$ be a curve in $G$ with $g(0)=e$, and $\Sigma\subset M$ be a minimal hypersurface so that $g(t)\cdot \Sigma\neq \Sigma$ for $t\neq 0$. 
Then the variation $g(t)\cdot \Sigma$ corresponds to a non-trivial Jacobi field of $\Sigma$ since ${\rm Area}(g(t)\cdot \Sigma) \equiv {\rm Area}(\Sigma)$. 
Now, a natural question is: can the area functional be an `equivariant Morse-Bott function' on the space of closed embedded hypersurfaces so that every minimal hypersurface can only be degenerate due to the given symmetry?
More precisely, we have the following question:
\begin{question}\label{Q: equivariant Morse-Bott}
	Is it true that for a generic $G$-invariant Riemannian metric on $M^{n+1}$, every Jacobi field $X\in\mathfrak{X}^\perp(\Sigma)$ of a closed embedded minimal hypersurface $\Sigma$ (could be non-equivariant) has the form of $X(p)=(\frac{d}{dt}\big|_{t=0} g(t)\cdot p)^\perp$, where $p\in\Sigma$, $g(t)\subset G$ is a curve with $g(0)=e$, and $\perp$ denotes the component that is normal to $\Sigma$. 
\end{question}

When $G$ is finite, the answer to Question \ref{Q: equivariant Morse-Bott} is positive by a natural generalization of the bumpy metrics theorem in \cite{white2017bumpy}. 
Indeed, White has shown in \cite[Theorem 2.1]{white2017bumpy} that a generic $G$-invariant Riemannian metric on $M$ is {\em bumpy} in the following sense: no finite cover of any closed, almost embedded, minimal submanifold $\Sigma$ admits a non-trivial Jacobi field. 
Here, the key result is the minimal submanifolds and the Jacobi fields need not be $G$-invariant. 
Note the finiteness of $G$ is used to ensure that $G\cdot\Sigma$ is still an almost embedded submanifold with the same dimension as $\Sigma$.

When $G$ is a general compact Lie group, a weak result of Question \ref{Q: equivariant Morse-Bott} will be sufficient to use for our purpose since we only consider $G$-invariant minimal hypersurfaces in Theorem \ref{Thm: 1}. 
Specifically, we say a Riemannian metric is {\em $G$-bumpy} if no finite cover of any closed embedded $G$-invariant minimal submanifold admits a non-trivial Jacobi field  (Definition \ref{Def: G-bumpy}). 
Then we build the following $G$-bumpy metrics theorem:

\begin{theorem}[$G$-Bumpy Metrics Theorem]\label{Thm: bumpy metric}
	Let $M^{n+1}$ be a closed smooth manifold with a compact Lie group $G$ acting as diffeomorphisms of cohomogeneity ${\rm Cohom}(G)\geq 2$. 
	If $k$ is an integer no less than $3$, or $k=\infty$. 
	Then for a generic $G$-invariant $C^k$ Riemannian metric $\gamma$ on $M$, $\gamma$ is $G$-bumpy in the sense of Definition \ref{Def: G-bumpy}. 
\end{theorem}

We mention that in the definition of $G$-bumpy metrics, embedded minimal $G$-submanifolds can be replaced by immersed minimal submanifolds $\Sigma$ whose stabilizer $G_\Sigma=\{g\in G: g\cdot \Sigma=\Sigma\}$ has finite index $[G:G_\Sigma]<\infty$, and the above theorem is still valid. 
Besides, we note that the set of equivariant Morse functions is dense in the space of smooth equivariant functions by \cite{wasserman1969equivariant}. 
In light of this, we hope the answer to Question \ref{Q: equivariant Morse-Bott} is affirmative, which will make it possible to apply equivariant Morse theory to investigate the space of minimal hypersurfaces. 


\subsection{Compactness and generic finiteness of min-max $G$-hypersurfaces}
In a three-dimensional manifold $M$, the compactness for minimal surfaces with bounded topology was developed by Choi-Schoen \cite{choi1985space} provided $\Ric_M>0$.
With area and topology bound assumptions, this result was later extended by White \cite{white1987curvature} for general parametric elliptic functionals. 
In higher dimensions, Sharp derived a compactness theorem \cite[Theorem 2.3]{sharp2017compactness} for minimal hypersurfaces with area and Morse index uniformly bounded, which is based on the curvature estimate for stable minimal hypersurfaces of Schoen-Simon \cite{schoen1981regularity}. 
Additionally, the method developed in \cite{simon1987strict}\cite{colding2000embedded} also suggests the existence of a non-trivial Jacobi field on the limit hypersurface. 
Hence, combining \cite[Theorem 2.3]{sharp2017compactness} with \cite[Theorem 2.1]{white2017bumpy}, one immediately obtains a generic finiteness result for closed minimal hypersurfaces.

Under equivariant constraints, we wonder how to derive curvature estimates from the characteristics of equivariant variations. 
By \cite[Lemma 7]{wang2022min}, if a minimal $G$-hypersurface $\Sigma$ admits a $G$-invariant unit normal, then the $G$-stability (Definition \ref{Def: G-index G-stable}) of $\Sigma$ is equivalent to the stability, which further suggests the curvature estimate by \cite{schoen1981regularity}. 
Hence, a natural idea is to search locally on $\Sigma$ for a $G$-invariant unit normal. 

Driven by this idea, we see from the triviality of the slice representation for principal orbits \cite[Corollary 2.2.2]{berndt2016submanifolds} that there is a small neighborhood $U$ of any given principal orbit $G\cdot p\subset\Sigma\cap M^{prin}$ so that $U$ is separated by $\Sigma$ into two open $G$-subset, which immediately gives a $G$-invariant unit normal on $\Sigma\cap U$. 
On a non-principal orbit, however, the issue is more subtle since the $G$-invariant normal vector fields of $\Sigma$ may all vanish on that orbit, especially for some $\mZ_2$-actions (cf. \cite[Remark 4.7]{ketover2016equivariant}). 

One potential approach is to study the intersection between $G$-hypersurfaces and orbit type strata. 
For instance, if the intersection between a $G$-hypersurface $\Sigma$ and each non-principal orbit type stratum is always orthogonal, then one can show the local existence of a $G$-invariant unit normal on $\Sigma$  (cf. \cite[Lemma 4.3]{ketover2016equivariant}). 
In another upcoming paper, we will elaborate on this idea and demonstrate the compactness theorem for minimal $G$-hypersurfaces with bounded area and $G$-index. 

Nevertheless, the minimal $G$-hypersurfaces concerned in this paper are mainly constructed by the equivariant min-max theory, i.e. {\em min-max $G$-hypersurfaces} (Definition \ref{Def: min-max G-hypersurfaces}). 
Therefore, the regularity of such min-max $G$-hypersurfaces comes from the {\em good $G$-replacement property in annuli} (Definition \ref{Def: good replacements property}), which holds under varifolds convergences through a process of taking subsequence. 
For the reader's convenience, we give a brief description of these concepts (see Definition \ref{Def: good replacements property}, \ref{Def: amv in admissible annuli}, \ref{Def: min-max G-hypersurfaces}). 
\begin{definition}[Good $G$-replacement property, min-max $(c,G)$-hypersurfaces]\label{Def: good replacements property and min-max G-hypersurfaces}
	Let $U\subset M$ be an open $G$-set and $V_0\in\V^G_n(M)$ be a $G$-varifold that is stationary in $U$. 
	We say $V_0$ has {\em good $G$-replacement property} in $U$, if for any finite sequence of compact $G$-sets $\{K_i\subset U\}_{i=1}^q$, there are $G$-varifolds $\{V_i\}_{i=1}^q\subset\V_n^G(M)$ stationary in $U$ so that 
	\[V_{i-1} \llcorner (M\setminus K_i) = V_i \llcorner (M\setminus K_i),\quad \|V_0\|(M) = \|V_i\|(M), \quad  V_i\llcorner \interior (K_i)= |\Sigma_i|,  \]
	where $1\leq i\leq q$ and $\Sigma_i$ is a stable minimal $G$-hypersurface with integer multiplicity. 
	
	Then a minimal $G$-hypersurface $\Sigma$ (possibly with multiplicity in $\N$) is said to be a {\em min-max $(c,G)$-hypersurface (with multiplicity)} for some $c\in\mZ_+$, if for any $c$ concentric $G$-annuli $\mathcal{A}=\{\an (p,s_i,t_i)\}_{i=1}^c$ in $M$ with $s_i<t_i<\frac{1}{2}s_{i+1} < \frac{1}{2}t_{i+1}$, $i\in\{1,\dots,c-1\}$, $|\Sigma|$ has good $G$-replacement property in at least one of $\an\in\mathcal{A}$. 
\end{definition}
Note the equivariant Morse index of minimal $G$-hypersurfaces can be generalized to the sense of varifolds (Definition \ref{Def: unstable varifolds}). 
Together, we build the following compactness theorem for min-max $G$-hypersurfaces (see Theorem \ref{Thm: compactness theorem}, \ref{Thm: Jacobi field}, Proposition \ref{Prop: index estimate}). 

\begin{theorem}[Compactness theorem for min-max $G$-hypersurfaces]\label{Thm: 3}
	Let $(M^{n+1},g_{_M})$ be a closed Riemannian manifold and $G$ be a compact Lie group acting by isometries on $M$ so that $3\leq {\rm codim}(G\cdot p) \leq 7$ for all $p\in M$. 
	Suppose $c\in\mZ_+$, $\{\Sigma_i \}_{i\in\N}$ is a sequence of min-max $(c,G)$-hypersurfaces (Definition \ref{Def: min-max G-hypersurfaces}) with $\sup_{i\in\N} \|\Sigma_i\|(M) \leq C <+\infty$. 
	Then $\Sigma_i $ converges (up to a subsequence) in the varifold sense to a min-max $(c,G)$-hypersurface $\Sigma$ with a constant multiplicity on each $G$-component. 
	Moreover, 
	\begin{itemize}
		\item[(i)] the convergence (up to a subsequence) is locally smooth and graphical with multiplicity on $\Sigma$ except for a finite union of orbits $\mathcal{Y}=\cup_{k=1}^K G\cdot p_k$;
		\item[(ii)] if the multiplicity of the local smooth graphical convergence is $1$, then $\mathcal{Y}=\emptyset$;   
		\item[(iii)] if $\Sigma_i\neq\Sigma$ eventually, then there is a non-trivial $G$-invariant Jacobi field on $\Sigma$ or its double cover; 
		 \item[(iv)] if $\sup_{i\in\N}{\rm index}_G(\Sigma_i) \leq I$, then ${\rm Index}_G(\Sigma)\leq I$.
	\end{itemize}
\end{theorem}

As a corollary, Theorem \ref{Thm: bumpy metric} and \ref{Thm: 3} immediately lead to the following $C^\infty_G$-generic finiteness result for min-max $(c,G)$-hypersurfaces with bounded area. 
\begin{corollary}\label{Cor: generic finiteness}
	Suppose a compact Lie group $G$ acts on a closed manifold $M$ by diffeomorphisms with $3\leq {\rm codim}(G\cdot p) \leq 7$ for all $p\in M$. 
	Then for any $c\in\mZ_+$ and a $C^\infty$-generic $G$-invariant Riemannian metric $g_{_M}$ on $M$, the set of min-max $(c,G)$-hypersurfaces in $(M,g_{_M})$ with mass uniformly bounded is finite. 
\end{corollary}

Combining Theorem \ref{Thm: 3}, Corollary \ref{Cor: generic finiteness}, with an equivariant deformation argument generalized from \cite{marques2016morse}, we have the conclusion of Theorem \ref{Thm: 1}. 
As a further application, we see that for every $k\in\mZ^+$, the {\em $(G,k)$-width} of $M$ defined in (\ref{Eq: Gp-width}) can be achieved by the area of some minimal $G$-hypersurfaces (possibly with multiplicities) with $G$-index bounded by $k$.
\begin{corollary}\label{Cor: achieve width of M}
	Suppose a compact Lie group $G$ acts on a closed Riemannian manifold $(M,g_{_M})$ by isometries with $3\leq {\rm codim}(G\cdot p) \leq 7$ for all $p\in M$. 
	Then for any $k\in \mZ^+$, there exists a closed smooth embedded minimal $G$-hypersurface $\Sigma \in \IV_n^G(M)$ with possibly multiplicity such that $\|\Sigma\|(M)=\omega^G_k(M,g_{_M})$ and ${\rm Index}_G(\spt(\|\Sigma\|)) \leq k$. 
\end{corollary}

\subsection{Outline}
In Section \ref{Sec: preliminary}, we introduce some notations and collect some useful lemmas. 
In Section \ref{Sec: generic metric}, we generalize the bumpy metrics theorem \cite{white2017bumpy} of White into our equivariant setting. 
In Section \ref{Sec: equivariant min-max}, we introduce the equivariant min-max constructions, and the compactness result for min-max $G$-hypersurfaces is given in Section \ref{Sec: compactness}. 
Finally, in Section \ref{Sec: G-index estimates}, we apply the deformation program to show the $G$-index estimates.

\noindent
{\bf Acknowledgement.} The author would like to thank Prof. Gang Tian for his constant support. 
He also thanks Zhi'ang Wu and Giada Franz for helpful discussions, and thanks the anonymous referees for helpful comments. 
The author is supported by the China post-doctoral grant 2022M722844. 

\noindent
{\bf Funding.} Project funded by China Postdoctoral Science Foundation 2022M722844. 


\section{Preliminary}\label{Sec: preliminary}

Let $(M^{n+1}, g_{_M})$ be an orientable connected closed Riemannian $(n+1)$-dimensional manifold and $G$ be a compact Lie group acting as isometries on $M$. 
Let $\mu$ be a bi-invariant Haar measure on $G$ which has been normalized to $\mu(G)=1$. 
It is well-known that there exists a (unique) minimal conjugate class $(P)$ of isotropy groups, namely the {\em principal orbit type}, such that the union of all points $p\in M$ with $(G_p)=(P)$, i.e. the union of principal orbits $M^{prin}$, forms an open dense submanifold of $M$ (\cite[Proposition 2.2.4]{berndt2016submanifolds}). 
%
%
%
The {\em cohomogeneity} ${\rm Cohom}(G)$ is defined by the co-dimension of a principal orbit $G\cdot p\subset M^{prin}$, i.e. ${\rm Cohom}(G) = \min_{p\in M} {\rm codim}(G\cdot p)$.

Denote by $B_r(p)$ and $\mathbb{B}^k_r(p)$ the geodesic ball in $M$ and the Euclidean ball in $\R^k$ respectively. 
If $A\subset M$ is a subset (submanifold, hypersurface, ...) with $G\cdot A = A$, then $A$ is said to be a {\em $G$-set} ({\em $G$-submanifold}, {\em $G$-hypersurface}, ...). 
Similarly, if a vector field $X$ on $M$ satisfies $dg(X)=X$ for all $g\in G$, where $dg$ is the tangent map of $g$, then $X$ is said to be a {\em $G$-vector field}. 
 We also use the following notations:
 \begin{itemize}
		\item $B_\rho^G(p)$: open geodesic tubes with radius $\rho$ around the orbit $G\cdot p$ in $M$;
		\item $\an(p,s,t)$: the open tube $B_t^G(p)\setminus \overline{B}^G_s(p)$; 
		\item $\mathfrak{X}(M), \mathfrak{X}^G(M)$: the space smooth vector fields and $G$-vector fields on $M$. 
\end{itemize}

\subsection{Notations in geometric measure theory}

In this subsection, we introduce some notations in geometric measure theory, which are referenced from \cite{pitts2014existence}\cite{simon1983lectures}. 
Let 
\begin{itemize}
	\item ${\bf I}_k(M;\mathbb{Z}_2)$ be the space of $k$-dimensional mod $2$ flat chains in $\mathbb{R}^L$ with support contained in $M$; 
	\item ${\mathcal Z}_n(M;\mathbb{Z}_2)$ be the space of flat chains $T\in {\bf I}_n(M;\mathbb{Z}_2)$ such that there exists $U\in {\bf I}_{n+1}(M;\mathbb{Z}_2)$ with $\partial U = T$, i.e. the boundary type mod $2$ $n$-cycles; 
	\item $\mathcal{V}_k(M)$ be the closure, in the weak topology, of the space of $k$-dimensional rectifiable varifolds in $\mathbb{R}^L$ with support contained in $M$. 
\end{itemize}
Denote by $\mathcal{IV}_k(M)\subset \mathcal{V}_k(M)$ the space of integral varifolds in $\mathcal{V}_k(M)$. 

Given $T\in {\bf I}_k(M;\mathbb{Z}_2)$, let $|T|$ and $\|T\|$ be the integral varifold and the Radon measure induced by $T$. 
Given $V\in \mathcal{V}_k(M)$, we denote by $\|V\|$ the Radon measure induced by $V$. 
Additionally, the {\it flat (semi-)norm} $\mathcal{F}$ and the {\it mass} ${\bf M}$ for flat chains in ${\bf I}_k(M;\mathbb{Z}_2)$ are defined in \cite[4.2.26]{federer2014geometric}. 
The ${\bf F}$-{\it metric} on $\mathcal{V}_k(M)$ is defined as in \cite[Page 66]{pitts2014existence}, which induces the varifold weak topology on any mass bounded subset of $\V_k(M)$. 
Then we use $\mathbf{B}^{\mF}_r(V) $ and $\overline{\mathbf{B}}^{\mF}_r(V)$ to denote the open and closed metric balls in $\mathcal{V}_k(M)$. 
Finally,  the ${\bf F}$-{\it metric} on ${\bf I}_k(M;\mathbb{Z}_2)$ is defined by
\[ {\bf F}(S,T):=\mathcal{F}(S-T)+{\bf F}(|S|,|T|).\]
For simplicity, we use the notations ${\bf I}_k(M;{\bf v};\mathbb{Z}_2)$, ${\mathcal Z}_n(M;{\bf v};\mathbb{Z}_2)$ to indicate that the topology are induced by the metric ${\bf v}$, where ${\bf v}={\bf M}$, ${\bf F}$, or $\F$. 

Using the push forward of the $G$ action, we say a varifold $V\in \mathcal{V}_k(M)$ is a {\em $G$-varifold} if $g_\#V=V$ for all $g\in G$. 
Similarly, a current $T$ with $\mZ_2$ coefficients is said to be a {\em $G$-current} if $g_\#T=T$ for all $g\in G$. 
Moreover, we can define the following subspaces of $G$-invariant elements: 
\begin{itemize}
	\item ${\bf I}_k^G(M;\mathbb{Z}_2) := \{T\in {\bf I}_k(M;\mathbb{Z}_2) : ~g_\#T= T,~\forall g\in G \}$;
	\item ${\mathcal Z}_n^G(M;\mathbb{Z}_2) := \{T\in {\mathcal Z}_n(M;\mathbb{Z}_2) : ~T=\partial U,~{\rm for~some~} U\in {\bf I}_{n+1}^G(M;\mathbb{Z}_2)  \}$; 
	\item $\mathcal{V}_k^G(M) := \{V\in \mathcal{V}_k(M) : ~g_\#V=V,~\forall g\in G\}$;
\end{itemize}
and $\mathcal{IV}_k^G(M) := \mathcal{IV}_k(M)\cap \mathcal{V}_k^G(M)$. 

Moreover, the norms and metrics $\M,\F,\mF$ can be naturally induced to the subspaces ${\bf I}_k^G(M;\mathbb{Z}_2)$, ${\mathcal Z}_n^G(M;\mathbb{Z}_2)$, $\mathcal{V}_k^G(M)$. 
 Indeed, since $G$ acts by isometries, we have $\M(T)=\M(g_\#T)$ and $\F(T)=\F(g_\#T)$ for all $g\in G$, $T\in {\bf I}_k(M;\mathbb{Z}_2)$. 
Therefore, ${\bf I}_k^G(M;{\bf v};\mathbb{Z}_2) $ and ${\mathcal Z}_n^G(M;{\bf v};\mathbb{Z}_2)$ are closed subspaces of ${\bf I}_k(M;{\bf v};\mathbb{Z}_2)$ and ${\mathcal Z}_n(M;{\bf v};\mathbb{Z}_2)$ respectively. 
Meanwhile, if $\{V_i\}_{i\in\N}\subset\V_k^G(M)$ is a sequence of $G$-varifolds so that $\lim_{i\to\infty}\mF(V_i,V)=0$ for some $V\in\V_k(M)$. 
Then, for every $g\in G$, we have $\mF(V_i,V) =\mF(g_\#V_i, g_\#V) =\mF(V_i, g_\#V) \to 0$ as $i\to\infty$, which implies $V\in \V_k^G(M)$ and $\V_k^G(M)\subset \V_k(M)$ is closed.

If $Q$ is an isoperimetric choice of $T_1,T_2\in \mI_n(M;\mathbb{Z}_2)$, $\partial T_i=0$, in the sense of \cite[Corollary 1.14]{almgren1962homotopy}, then $g_\#Q$ is an isoperimetric choice of $g_\#T_1, g_\#T_2$. 
By the uniqueness result \cite[Lemma 3.1]{marques2017existence}, we have the following lemma showing the isoperimetric choice of $G$-invariant codimension-one cycles is also $G$-invariant, and thus $\Z_n^G(M;\mZ_2)$ is indeed the $\F$-path connected component of $\Z_n(M;\mZ_2)\cap \mI_n^G(M;\mZ_2)$ containing $0$. 

\begin{lemma}\label{Lem: isoperimetric}
	There exist constants $\nu_M>0$ and $C_M>1$ such that for any $T_1,T_2\in {\bf I}_n^G(M;\mathbb{Z}_2)$ with $\partial T_1=0$, $\partial T_2= 0$, and 
	\[\mathcal{F}(T_1-T_2)<\nu_M, \]
	there exists a unique $Q\in {\bf I}_{n+1}^G(M;\mathbb{Z}_2)$ such that
	\begin{itemize}
    	\item $\partial Q = T_1-T_2$;
    	\item ${\bf M}(Q)\leq C_M\cdot\mathcal{F}(T_1-T_2)$.
  	\end{itemize}
\end{lemma}
\begin{proof}
	See \cite[Lemma 5]{wang2022min}. 
\end{proof}

Note every $T\in {\mathcal Z}_n^G(M;\mathbb{Z}_2)$ is a boundary of some $U\in {\bf I}_{n+1}^G(M;\mathbb{Z}_2)$. 
If there is another $V\neq U \in  {\bf I}_{n+1}^G(M;\mathbb{Z}_2)$ so that $\partial V=\partial U = T$, then we have $V= M-U$ by the Constancy Theorem for mod $2$ flat chains. 
Hence, the boundary map
\[\partial ~:~ {\bf I}_{n+1}^G(M;\mathbb{Z}_2) \rightarrow {\mathcal Z}_n^G(M;\mathbb{Z}_2)\]
is a double cover. 
Moreover, using the equivariant Morse function, we can show that ${\bf I}_{n+1}^G(M;\mathbb{Z}_2)$ is contractible, and the boundary map $\partial$ satisfies the lifting property. 
Indeed, we have the following result parallel to \cite[Theorem 5.1]{marques2021morse}. 

\begin{proposition}\label{Prop: isomorphism}
	The space of $G$-boundary cycles ${\mathcal Z}_n^G(M;\mathbb{Z}_2)$ is weakly homotopically equivalent to $\mathbb{RP}^\infty$, i.e. $\pi_1({\mathcal Z}_n^G(M;\mathbb{Z}_2),0)=\mZ_2$, and $\pi_k({\mathcal Z}_n^G(M;\mathbb{Z}_2),0)=0$ for $k\geq 2$.  
\end{proposition}
\begin{proof}
	By the equivariant Morse theory, there exists a $G$-equivariant Morse function $f:M\to [0,1]$ in the sense of \cite{wasserman1969equivariant}. 
	Hence, the proof of \cite[Theorem 5.1]{marques2021morse} would carry over with the $G$-Invariant Isoperimetric Lemma \ref{Lem: isoperimetric} in place of \cite[Corollary 1.14]{almgren1962homotopy}. 
\end{proof}

\subsection{Equivariant variations}

Recall a varifold $V\in\V_k(M)$ is said to be {\em stationary} in an open set $U\subset M$ if 
\[\delta V(X):= \frac{d}{dt}\Big|_{t=0} \|(\phi_t)_\# V\| (M) =\int {\rm div}_SX(x)dV(x,S) =0,\]
for all $X\in\mathfrak{X}(M)$ compactly supported in $U$, where $\{\phi_t\}$ are the diffeomorphisms generated by $X$. 
In the equivariant case, we only need to consider the variations generated by $X\in \mathfrak{X}^G(M)$ and have the following definition:
\begin{definition}
	Let $V\in \V_k^G(M)$ be a $G$-varifold and $U\subset M$ be an open $G$-set. 
	Then $V$ is said to be {\em $G$-stationary in $U$} if $\delta V(X) = 0$ for all $X\in \mathfrak{X}^G(M)$ compactly supported in $U$. 
\end{definition}

For any $X\in \mathfrak{X}(M)$, define the averaged vector field of $X$ by 
\begin{equation}\label{Eq: average vector field}
	X_G = \int_G d(g^{-1})(X) d\mu(g),
\end{equation}
which is a $G$-invariant vector field since 
\[dh(X_G) = \int_G d(h\circ g^{-1}) (X) d\mu(g) = \int_G d((g\circ h^{-1})^{-1}) (X) d\mu(g\circ h^{-1}) = X_G,\quad \forall h\in G.\]
Using this averaging trick, we have the following lemma (see \cite[Lemma 3.2]{liu2021existence}). 
\begin{lemma}\label{Lem: G-stationary}
Given any $G$-varifold $V\in \V_k^G(M)$ and an open $G$-set $U\subset M$, $V$ is stationary in $U$ if and only if $V$ is $G$-stationary in $U$. 
\end{lemma}

Suppose $\Sigma\subset M$ is a closed minimal hypersurface, i.e. $\delta\Sigma (X) = 0,~\forall X\in \mathfrak{X}(M)$. 
Let $X\in \mathfrak{X}^\bot (\Sigma)$, where $\mathfrak{X}^\bot (\Sigma)$ is the space of normal vector fields on $\Sigma$ with compact support. 
After extending $X$ to a compactly supported vector field $\widetilde{X}\in\mathfrak{X}(M)$, we have diffeomorphisms $\{\phi_t\}_{t\in(-1,1)}$ generated by $\widetilde{X}$. 
Then the second variation formula of $\Sigma$ is written by: 
\[\delta^2\Sigma(X) = \frac{d^2}{dt^2}\Big|_{t=0}{\rm Area}(\phi_t(\Sigma)) = -\int_\Sigma \langle L_\Sigma X,X \rangle, \]
where $L_\Sigma: \mathfrak{X}^\bot (\Sigma)\to \mathfrak{X}^\bot (\Sigma)$ is the Jacobi operator of $\Sigma$. 
Recall the {\em Morse Index ${\rm Index}(\Sigma)$ of $\Sigma$} is defined as the number of the negative eigenvalues of the Jacobi operator $L_\Sigma$ (counted with multiplicities). 
If $\delta^2\Sigma(X) \geq 0$ for all $X\in \mathfrak{X}^\bot (\Sigma)$, i.e. ${\rm Index}(\Sigma)=0$, then $\Sigma$ is said to be {\em stable}.

Additionally, suppose $\Sigma$ is $G$-invariant. 
Define then 
\[\mathfrak{X}^{\bot,G} (\Sigma) :=\{X\in  \mathfrak{X}^\bot (\Sigma): dg(X)=X,\forall g\in G\}\]
as the space of normal $G$-vector fields on $\Sigma$. 
For any $X\in \mathfrak{X}^{\bot,G}(\Sigma)$ with a compactly supported extension $\widetilde{X}\in\mathfrak{X}(M)$, define $\widetilde{X}_G$ to be the averaged vector field of $\widetilde{X}$ given as in (\ref{Eq: average vector field}). 
Then, $\widetilde{X}_G\in\mathfrak{X}^G(M)$ is an equivariant extension of $X$. 
Moreover, since $G$ acts by isometries and $\Sigma$ is $G$-invariant, we have $L_\Sigma X\in \mathfrak{X}^{\bot,G} (\Sigma)$ for all $X\in \mathfrak{X}^{\bot,G} (\Sigma)$. 
Hence, we can restrict the Jacobi operator on the subspace $\mathfrak{X}^{\bot,G} (\Sigma)$, which gives the following definition:

\begin{definition}\label{Def: G-index G-stable}
	Let $\Sigma\subset M$ be a smooth closed embedded $G$-invariant minimal hypersurface. 
	The {\em equivariant Morse index} (or {\em $G$-index} for simplicity) ${\rm Index}_G(\Sigma)$ is defined as the number of negative eigenvalues (counted with multiplicities) of the Jacobi operator $L_\Sigma$ acting on the subspace $\mathfrak{X}^{\bot,G} (\Sigma)$. 
	If $\delta^2\Sigma(X)\geq 0$ for all $X\in \mathfrak{X}^{\bot,G} (\Sigma)$, then we say $\Sigma$ is {\em $G$-stable}. 
\end{definition}

Clearly, $\Sigma$ is $G$-stable if and only if ${\rm Index}_G(\Sigma)=0$. In Section \ref{Sec: G-index estimates}, we will generalize the $G$-index for smooth embedded minimal $G$-hypersurfaces to the varifold sense.

\section{Equivariant generic metrics}\label{Sec: generic metric}

In this section, we extend the bumpy metrics theorem of White \cite{white2017bumpy} to a more general case of compact Lie group actions. 
First, we have the following structure theorem, which is parallel to \cite[Theorem 2.3]{white2017bumpy}. 
To be consistent with our previous notations, we exchange the 
symbols `$N$' and `$M$' in \cite{white2017bumpy}. 
\begin{theorem}[Structure Theorem]\label{Thm: structure theorem}
	Let $M^{n+1}$ be a connected closed smooth manifold with a compact Lie group $G$ acting as diffeomorphisms. 
	Suppose ${\rm Cohom}(G)\geq 2$. 
	Given any integer $k\geq 3$, let $\Gamma_k$ be the space of $G$-invariant  $C^k$ Riemannian metrics on $M$. 
	Denote $\mathcal{N}_k$ as the space of pairs $(\gamma, N)$ where $\gamma\in\Gamma_k$ and $N$ is a closed embedded $G$-invariant $\gamma$-minimal hypersurface in $M$. 
	Then $\mathcal{N}_k$ is separable $C^{k-2}$ Banach manifold, and the map 
	\begin{eqnarray*}
		&& \pi : \mathcal{N}_k\rightarrow \Gamma_k
		\\
		&& \pi(\gamma, N) = \gamma
	\end{eqnarray*} 
	is a $C^{k-2}$ Fredholm map of Fredholm index $0$. 
	Additionally, $(\gamma, N)$ is a critical point of $\pi$ if and only if $N$ has a nontrivial $G$-invariant Jacobi field. 
\end{theorem}
\begin{proof}
	Firstly, fix the diffeomorphism type of $N$. 
	Then we only need to show an equivariant version of \cite[Theorem 2.1]{white1991space} for compact Lie group $G$. 
	Indeed, once we replace the notations `metric', `manifold', and `embedding' in \cite[Theorem 2.1]{white1991space} with `$G$-invariant metric', `$G$-manifold', and `$G$-equivariant embedding', the proof of \cite[Theorem 2.1]{white1991space} would carry over. 
	Here, we only list some necessary modifications. 
	
	By \cite[Corollary 1.12]{wasserman1969equivariant}, any equivariant $C^{2,\alpha}$ embedding $w_0: N\to M$ can be approximated by a $C^\infty$ equivariant map $\iota\in C^\infty(N,M) $. 
	Note all the proper $C^2$ embeddings form an open subset in $C^2(N, M)$ in the Whitney topology (\cite[Chapter 2, Theorem 1.4]{hirsch1976differential}). 
	Hence, $\iota$ is also an embedding. 
	Let $V$ be the normal bundle of $\iota$ over $N$ with $G$ acting on it by the tangent maps. 
	Then the normal exponential map $E: V\to M$ is a $G$-equivariant diffeomorphism in a neighborhood of $\iota(N)$. 
	Thus, for $\iota$ sufficiently close to $w_0$, there is a $G$-equivariant section $u_0\in C^{2,\alpha}(N,V)$ of $V$ so that $w_0 = E\circ u_0$. 
	This is parallel to the first paragraph in the proof of \cite[Theorem 2.1]{white1991space}. 
	
	Since the Riemannian metric $\gamma$ is $G$-invariant, i.e. $G$ acts by isometries, the $\gamma$-area element functional $A_\gamma(x, u(x), Du(x))$ for any equivariant embedding $E\circ u$ defined in \cite[Page 177]{white1991space} is also $G$-invariant. 
	Thus, by the first variation formula \cite[Page 178]{white1991space}, we have a map $H: \Gamma_k \times C^{2,\alpha}_G(N,V)\to C^{0,\alpha}_G(N,V)$ so that $u\in C^{2,\alpha}_G(N,V)$ is $\gamma$-stationary if and only if $H(\gamma, u)=0$, where $C^{j,\alpha}_G(N,V)$ is the set of $G$-equivariant maps in $C^{j,\alpha}(N,V)$. 
	To apply \cite[Theorem 1.1, 1.2]{white1991space}, we only consider the subspace of equivariant maps $ C^{2,\alpha}_G(N,V)$, $ C^{0,\alpha}_G(N,V)$, $L^2_G(N,V)$. 
	Additionally, we can apply the averaging trick to the function $f$ constructed in \cite[Page 179]{white1991space}, which gives $\hat{f}(x) = \int_G f(g\cdot x)d\mu(g)$ and $\nabla\hat{f} = \int_G d(g^{-1}) (\nabla f) d\mu(g) $. 
	By the $G$-invariance of $\gamma_0$, $E$, and $\kappa\in {\ker}D_2H(\gamma_0, u_0)$, one can verify that $\gamma_s(z) = (1+s\hat{f}(z))\gamma_0(z) \in \Gamma_k$ satisfies \cite[Page 178 (C)]{white1991space}. 
	Indeed, it is sufficient to take a $G$-invariant function $\hat{f}$ with $\nabla\hat{f}(E(x,u_0(x)))$ given by the $G$-invariant vector field $D_2E(x,u_0(x))\kappa(x)$. 
	Now, combining the above modifications with the general results \cite[Theorem 1.1, 1.2]{white1991space}, the proof of \cite[Theorem 2.1]{white1991space} would carry over. 
	
	Finally, since there are only countably many diffeomorphism types of $N$, we see $\mathcal{N}_k$ has countably many components, and thus the theorem holds. 
\end{proof}

In what follows, let $\Gamma_k ,\mathcal{N}_k$ be as in Theorem \ref{Thm: structure theorem} and $k\geq 3$ be fixed. 
Denote $\mathcal{C}_{\mathcal{N}_k}$ as the set of critical points of $\pi$.  
For every $p\in\mZ^+$, define $\mathcal{S}^p_k$ by: 
\begin{equation}
	\mathcal{S}^p_k : =\left\{\begin{array}{l|l} (\gamma, N)\in \mathcal{N}_k & \begin{array}{l}\text {There exists a $p$-sheeted covering } N' \text{ of a connected} \\ \text {component of } N \text { so that } N' \text{ has a nontrivial Jacobi field. } \end{array}\end{array}\right\}. \nonumber
\end{equation}
We then have the following lemma, which is parallel to \cite[Lemma 2.6]{white2017bumpy}. 
\begin{lemma}\label{Lem: bumpy}
	$\mathcal{S}^p_k$ is a closed subset of $\mathcal{N}_k$, and $\mathcal{S}^p_k\setminus\mathcal{C}_{\mathcal{N}_k}$ is nowhere dense in $\mathcal{N}_k$. 
\end{lemma}
\begin{proof}
	Using $ F(y) = \int_G f(g(y)) d\mu(g)$ in \cite[Page 1152]{white2017bumpy}, the local metric perturbation argument in \cite[Lemma 2.6]{white2017bumpy} would carry over. 
\end{proof}

\begin{definition}[$G$-bumpy]\label{Def: G-bumpy}
	Let $\gamma$ be a $G$-invariant $C^k$ Riemannian metric on $M$. 
	Then we say $\gamma$ is {\em $G$-bumpy} if no finite cover of any closed embedded $G$-invariant minimal hypersurface in $M$ admits a non-trivial Jacobi field. 
\end{definition}
\begin{remark}\label{Rem: bumpy metrics for lower dimensional sub-mfd}
	We can also use minimal $G$-submanifolds of any dimension in the above definition, but it is sufficient to consider only the $G$-hypersurfaces for our purpose. 
	Additionally, although the hypersurface is required to be $G$-invariant in Definition \ref{Def: G-bumpy}, the Jacobi fields need not be $G$-invariant. 
\end{remark}

\subsection{Proof of Theorem \ref{Thm: bumpy metric}}
\begin{proof}[Proof of Theorem \ref{Thm: bumpy metric}]
	First, let $k\geq 3$ be an integer. 
	Suppose $\Sigma$ is a closed embedded $\gamma$-minimal $G$-hypersurface so that there is a positive integer $p$ and a $p$-sheeted covering of $\Sigma$ admitting a nontrivial Jacobi field. 
	Clearly, $(\gamma, \Sigma)\in \mathcal{S}^p_k$, and thus $\gamma\in \pi (\mathcal{S}^p_k)\subset \cup_{p} \pi(\mathcal{S}^p_k)$. 
	Combining Lemma \ref{Lem: bumpy} with \cite[Theorem 2.4]{white2017bumpy}, we have $\cup_{p} \pi(\mathcal{S}^p_k)$ is a meager subset of $\Gamma_k$, which gives the theorem for $k<\infty$. 
	Next, let $W_k\subset \Gamma_k$ be the set of $G$-invariant $G$-bumpy $C^k$ metrics. 
	Then we have $W_k\subset \Gamma_k$ is a second-category subset, and $W_\infty = W_3\cap \Gamma_\infty = \cap_{k\geq 3} (W_k\cap \Gamma_\infty)$. 
	Noting also $\mS_{k+1}^p= \mS_k^p\cap \mathcal{N}_{k+1}$, $\C_{k+1} = \C_k\cap \mathcal{N}_{k+1} $ and $W_k=\Gamma_k\setminus\cup_{p\geq 1} \pi(S_k^p) $, we can use \cite[Theorem 2.4, 2.10]{white2017bumpy}(see also \cite[Section 7.1]{ambrozio2018compactness} for specific constructions) to show $W_\infty\subset \Gamma_\infty$ is a second-category subset. 
\end{proof}

Let $\mathcal{M} $ be the set of closed smooth embedded minimal hypersurfaces in $M$, and $\mathcal{M}_G \subset \mathcal{M} $ contain all the $G$-invariant elements of $\mathcal{M} $. 
Given $I\in \N$ and $C>0$, define:
\begin{equation}\label{Eq: definition}
	\mathcal{M}_G(I,C) := \{ \Sigma\in \mathcal{M}_G ~:~  {\rm Index}_G(\Sigma)\leq I, ~\mathcal{H}^n(\Sigma)\leq C \}. 
\end{equation}
An application of Theorem \ref{Thm: bumpy metric} combined with the compactness theorem \cite[Theorem 2.3]{sharp2017compactness} shows the following corollary:
\begin{corollary}\label{Cor: finite hypersurfaces}
	Let $M^{n+1}$ be a closed smooth manifold with a compact Lie group $G$ acting as diffeomorphisms so that $3\leq n+1\leq 7$ and ${\rm Cohom}(G)\geq 2$. 
	Then for a generic $G$-invariant $C^\infty$ Riemannian metric $\gamma$ on $M$, the sets $\mathcal{M}_G$ and $\mathcal{M}_G(I,C)$ defined in (\ref{Eq: definition}) are both countable. 
	Additionally, this is also valid in arbitrary dimensional $G$-manifolds provided $3\leq {\rm codim}(G\cdot x)\leq7, \forall x\in M$. 
\end{corollary}
\begin{proof}
	By Theorem \ref{Thm: bumpy metric}, for a generic $G$-invariant $C^\infty$ Riemannian metric $\gamma$ on $M$, $\gamma$ is $G$-bumpy in the sense of Definition \ref{Def: G-bumpy}. 
	For fixed integers $C, I>0$, suppose there exist infinitely many $\gamma$-minimal $G$-invariant hypersurfaces $\{\Sigma_i\}_{i=1}^\infty$ with $\mathcal{H}^n(\Sigma_i)\leq C$ and ${\rm Index}(\Sigma_i)\leq I$ (not $G$-index). 
	Then, combining the compactness theorem \cite[Theorem 2.3]{sharp2017compactness} with the $G$-invariance of $\gamma$ and $\Sigma_i$,
	we see $\Sigma_i$ converges (up to a subsequence) in the varifold sense to a $G$-varifold $m\cdot \Sigma$, where $\Sigma$ is a $\gamma$-minimal $G$-hypersurface and $m\in\N$. 
	Moreover, \cite[Theorem 2.3]{sharp2017compactness} also implies the existence of a nontrivial Jacobi field on $\Sigma$ or its double cover, which contradicts the $G$-bumpy property of $\gamma$. 
	Meanwhile, if $3\leq {\rm codim}(G\cdot x)\leq7, \forall x\in M$, one can also use the $G$-invariance and the optimal regularity results (\cite{wickramasekera2014general}) in the proof of \cite[Lemma 4.1, Claim 1, 2]{sharp2017compactness} to show the limit $G$-varifold $m \Sigma =\lim \Sigma_i$ is a smooth $\gamma$-minimal $G$-hypersurface and the convergence is smooth except for at most {\em $I$ points}. 
	Then the Jacobi field constructions in \cite[Theorem 2.3]{sharp2017compactness} would carry over to give a contradiction. 
	
	Hence, the set $\mathcal{M}_{G,I,C}:=\{\Sigma\in \mathcal{M}_G :  {\rm Index}(\Sigma)\leq I, ~\mathcal{H}^n(\Sigma)\leq C\}$ is finite for the $G$-bumpy metric $\gamma$. 
	Finally, the corollary follows from $\mathcal{M}_G(I,C) \subset \mathcal{M}_G = \bigcup_{K,L=1}^{\infty} \mathcal{M}_{G,K,L}$. 
	%
\end{proof}

As we mentioned after Theorem \ref{Thm: bumpy metric}, our $G$-bumpy metric theorem can also be extended to some more general cases. 
For instance, Theorem \ref{Thm: bumpy metric} is also valid after replacing embedded minimal $G$-hypersurfaces in Definition \ref{Def: G-bumpy} by immersed minimal submanifolds $\Sigma$ whose stabilizer $G_\Sigma:=\{g\in G: g\cdot\Sigma=\Sigma\}$ has finite index $[G:G_\Sigma]$ in $G$. 
Indeed, by Remark \ref{Rem: bumpy metrics for lower dimensional sub-mfd}, the $G$-hypersurfaces can be replaced by $G$-submanifolds in the above results. 
Additionally, note $[G:G_\Sigma]<\infty$ indicates $G\cdot \Sigma$ is still an immersed minimal submanifold. 
Hence, combining with the arguments in \cite[Theorem 2.1, 2.9]{white2017bumpy}, the embeddedness and the $G$-invariance can be weaken to immersion and $[G:G_\Sigma]<\infty$ respectively. 

\section{Equivariant min-max theory}\label{Sec: equivariant min-max}

In this section, we introduce the equivariant min-max theory built in \cite{wang2022min} with some modifications.
Firstly, let $I^m = [0,1]^m$ be the $m$-dimensional unit cube. 
For any $j\in \N$, denote by $I(1,j)$ the cell complex with 
\[\mbox{$0$-cells $[0],[3^{-j}]\dots,[1-3^{-1}],[1]$, and $1$-cells $[0,3^{-j}],[3^{-j},2\cdot 3^{-j}],\dots,[1-3^{-j}, 1]$. }\]
Then the $m$-dimensional cubical complex is defined as $I(m,j):=I(1,j)^{\otimes m }$. 
For any $\alpha = \alpha_1\otimes\cdots\otimes\alpha_m\in I(m,j)$, we say $\alpha$ is a $p$-cell of $I(m,j)$ if $\sum_{i=1}^m\dim(\alpha_i) = p$, where each $\alpha_i$ is a cell of $I(1,j)$. 

Moreover, if $X$ is a subcomplex of dimension $k$ of $I(m,j)$ for some $m,j\in\N$, then we say $X$ is a {\em cubical complex} of dimension $k$. 
Let $X(n_i)$ be the union of all cells of $I(m, j+n_i)$ that are supported in some cell of $X$. 
Denote by $X(n_i)_q$ the set of all $q$-cells in $X(n_i)$. 

For any $i,j\in\N$, let ${\bf n}(i,j): X(i)_0\to X(j)_0$ be the nearest projection, i.e. ${\bf n}(i,j)(x)\in X(j)_0$ is the vertex that is closest to $x\in X(i)_0$. 
For any map $\phi: X(j)_0\to \Z_n^G(M;\mZ_2)$, the {\em fineness} of $\phi$ is defined by 
\[ {\bf f}_\M(\phi) := \sup\left\{{\bf M}(\phi(x)-\phi(y)) :  \{x,y\}=  \alpha_0, \alpha\in  X(j)_1 \right\}. \]

\subsection{Equivariant min-max setting}
\begin{definition}
	Let $X$ be a cubical subcomplex of $I^m$ and $\Phi_i:X\to {\mathcal Z}_n^G(M;\mF;\mathbb{Z}_2)$ be two continuous map, $i=1,2$. 
	We say $\Phi_1$ is {\em $G$-homotopic} to $\Phi_2$, if there exists a continuous map $\Psi:[0,1]\times X \to {\mathcal Z}_n^G(M;\F;\mathbb{Z}_2)$ so that $\Psi(i-1,\cdot)=\Phi_i$, $i=1,2$. 
	
	Moreover, for any continuous map $\Phi:X\to {\mathcal Z}_n^G(M;\mF;\mathbb{Z}_2)$, the {\em $G$-homotopy class} ${\bm \Pi}$ of $\Phi$ is the set of all $\Phi':X\to {\mathcal Z}_n^G(M;\mF;\mathbb{Z}_2)$ that are $G$-homotopic to $\Phi$. 
	Denote by $[X, {\mathcal Z}_n^G(M;\mF;\mathbb{Z}_2)]$ the set of all such $G$-homotopy classes. 
\end{definition}
Notice that the homotopy map in our definition is a mapping into the $G$-cycles space $ {\mathcal Z}_n^G(M;\mathbb{Z}_2)$ and is only $\F$-continuous. 

\begin{definition}
	Given any $G$-homotopy class ${\bm \Pi}\in [X, {\mathcal Z}_n^G(M;\mF;\mathbb{Z}_2)]$, the {\em width} of ${\bm \Pi}$ is defined as 
	\[\mathbf{L}({\bm \Pi}) := \inf_{\Phi\in {\bm \Pi}} \sup_{x\in X} \M(\Phi(x)).\]
	Additionally, a sequence $\{\Phi_i \}_{i\in\N}\subset {\bm \Pi}$ is said to be a {\em min-max sequence} if 
	\[\limsup_{i\to\infty}\sup_{x\in X}\M(\Phi_i(x))  =  \mathbf{L}({\bm \Pi}). \]
\end{definition}
\begin{definition}
	Given any $G$-homotopy class ${\bm \Pi}\in [X, {\mathcal Z}_n^G(M;\mF;\mathbb{Z}_2)]$ and a sequence $\{\Phi_i \}_{i\in\N}\subset {\bm \Pi}$, the {\em image set} of $\{\Phi_i \}_{i\in\N}$ is defined as
	\begin{eqnarray*}
		\mathbf{\Lambda} (\{\Phi_i \}_{i\in\N} ) := \{V \in \mathcal{V}_{n}^G(M): V=\lim_{j\to\infty} |\Phi_{i_{j}}(x_{i_{j}})| \mbox{~for some $i_j\to\infty,x_{i_j}\in X$}\}. 
	\end{eqnarray*}
	Moreover, if $\{\Phi_i \}_{i\in\N}$ is a min-max sequence, we define the {\em critical set} of $\{\Phi_i \}_{i\in\N}$ by 
	\[\mathbf{C}(\{\Phi_i \}_{i\in\N}) := \{V\in\mathbf{\Lambda} (\{\Phi_i \}_{i\in\N} ) : ||V||(M) = \mathbf{L}({\bm \Pi}) \}  .\]
\end{definition}

To apply the equivariant Almgren-Pitts min-max theory, we also need the following discrete homotopy analogies. 
\begin{definition}
	Let $X$ be a cubical subcomplex of $I^m$. 
	Given any two maps $\phi_i: X(n_i)_0 \to \Z_n^G(M;\mZ_2)$, $i=1,2$, we say 
	\[\mbox{{\em $\phi_1$ is homotopic to $\phi_2$ in $\Z_n^G(M;\M;\mZ_2)$ with fineness $\delta$}},\]
	if there exists another map $\psi: I(1, n_3)\times X(n_3)_0\to \Z_n^G(M;\mZ_2)$ so that $\mf_\M(\psi)<\delta$ and $\psi([i-1], \cdot) = \phi_i\circ {\bf n}(n_3, n_i)$, $i=1,2$. 
	
	Additionally, let $S=\{\varphi_i\}_{i\in\N}$ be a sequence of maps $\varphi_i : X(n_i)_0\to \Z_n^G(M;\mZ_2)$ so that $n_i\to\infty $ and $\mf_\M(\varphi_i)\to 0$. 
	Then we define 
	\begin{itemize}
		\item ${\bf L}(S):= \limsup_{i\to\infty} \max \{ \M(\varphi_i(x)) : x\in X(n_i)_0 \} $;
		\item $\mathbf{\Lambda}(S) := \{ V \in \mathcal{V}_{n}^G(M): V=\lim_{j\to\infty} |\varphi_{i_{j}}(x_{i_{j}})| \mbox{~for some $i_j\to\infty,x_{i_j}\in X(n_{i_j})_0$} \} $;
		\item $\mathbf{C}(S) := \{V\in\mathbf{\Lambda} (S) : ||V||(M) = \mathbf{L}(S) \} $.
	\end{itemize}
\end{definition}

By Lemma \ref{Lem: G-stationary}, for $C=2\mathbf{L}({\bm \Pi})$, we can apply the constructions in \cite[Page 150]{pitts2014existence} with $\mathfrak{X}^G(M)$ in place of $\mathfrak{X}(M)$ to obtain a continuous map 
\begin{eqnarray}\label{Eq-pulltightmap}
		H:~I\times \big(\mathcal{Z}_n^G(M;{\bf F};\mathbb{Z}_2)&\cap & \{T :  {\bf M}(T)\leq C \} \big)\nonumber
		\\
		&\rightarrow & \mathcal{Z}_n^G(M;{\bf F};\mathbb{Z}_2)\cap  \{T :  {\bf M}(T)\leq C \},
\end{eqnarray}
such that: 
	\begin{itemize}
		\item $H(0, T)= T$;
		\item if $|T|$ is stationary, then $H(t, T)= T$ for all $t\in[0,1]$;
		\item if $|T|$ is not stationary, then ${\bf M}(H(1, T)) < \M(T)$. 
	\end{itemize}

\begin{proposition}\label{Prop:pulltight}
	Let ${\bf \Pi}\in \big[ X, \Z_{n}^G(M;\mF;\mZ_2)\big]$ be a continuous $G$-homotopy class. 
	For any min-max sequence $\{\Phi_i^*\}_{i\in\N}$ in ${\bf \Pi}$, there exists a {\em pulled-tight} min-max sequence $\{\Phi_i\}_{i\in\N}$ of ${\bf \Pi}$ such that
	\begin{itemize}
		\item[(i)] ${\bf C}(\{\Phi_i\}_{i\in\N}) \subset {\bf C}(\{\Phi_i^*\}_{i\in\N})$;
		\item[(ii)] every $G$-varifold $V\in {\bf C}(\{\Phi_i\}_{i\in\N})$ is stationary in $M$.
	\end{itemize}
\end{proposition}
\begin{proof}
	For any $x\in X$, define $\Phi_i(x) = H(1, \Phi_i^*(x)) $. 
	Then the proposition follows from the properties of $H$. 
\end{proof}

After applying the discretization theorem (\cite[Theorem 2]{wang2022min}) to each $\Phi_i$ and passing to a diagonal subsequence, we can obtain a sequence $\{k_i\}_{i\in\N}\subset \N $ with $k_i\to\infty$, and a sequence $S=\{\varphi_i\}_{i\in\N}$ of discrete mappings 
\begin{equation}\label{Eq: discrete pull-tight}
	\varphi_i: X( k_i)_{0}\rightarrow \Z_{n}^G(M ;\mZ_2) 
\end{equation}
satisfying $\lim_{i\to\infty}\mf_\M(\varphi_i)= 0$, the Almgren $G$-extension $\Phi_i'$ of $\varphi_i$ (\cite[Theorem 3]{wang2022min}) is $G$-homotopic to $\Phi_i$, ${\bf L}(S)={\bf L}(\{\Phi_i\}_{i\in\N})={\bf L}({\bf \Pi})$, and ${\bf C}(S)={\bf C}(\{\Phi_i\}_{i\in\N})$. 
We refer to \cite[Proposition 1]{wang2022min} for the details of this discretized procedure.

\subsection{Almost minimizing varifolds}

Proposition \ref{Prop:pulltight} gives a stationary $G$-varifold as a weak solution. 
To obtain minimal $G$-hypersurfaces, let us introduce the following definitions, which generalized the definitions of Pitts in \cite{pitts2014existence}. 
\begin{definition}\label{Def:a.m.deform}
	Given any $\epsilon>0,~\delta>0$, open $G$-set $U\subset M$, and $T\in \mathcal Z_{n}^G(M;\mathbb{Z}_2)$, suppose $\{T_i\}_{i=1}^q\subset \mathcal Z_{n}^G(M;\mathbb{Z}_2)$ is a finite sequence such that 
	\begin{itemize}
		\item $T_0=T$ and ${\rm spt}(T-T_i)\subset U$ for all $i=1,\dots, q;$
		\item ${\bf v}(T_i-T_{i-1})\leq \delta$ for all $i=1,\dots, q;$
		\item ${\bf M}(T_i)\leq {\bf M}(T)+\delta$ for all $i=1,\dots, q;$
		\item ${\bf M}(T_q)< {\bf M}(T)-\epsilon$,
	\end{itemize}
	where ${\bf v}$ is any one of $\mathcal{F},~{\bf F},~{\bf M}$. 
	Then $\{T_i\}_{i=1}^q$ is said to be a {\em $G$-invariant $(\epsilon,\delta)$-deformation} of $T$ in $U$ under the metric ${\bf v}$. 
	Moreover, we define
	\[ \mathfrak{a}^G_n(U;\epsilon,\delta; {\bf v} ) \]
	to be the set of all $G$-cycles $T\in \mathcal Z_{n}^G(M;\mathbb{Z}_2)$ that do not admit any $G$-invariant $(\epsilon,\delta)$-deformation under the metric ${\bf v}$. 
\end{definition}

\begin{definition}\label{Def:a.m.}
	A varifold $V\in \mathcal V_{n}^G(M)$ is said to be {\em $(G,\mathbb{Z}_2)$-almost minimizing} in an open $G$-set $U\subset M$, if for every $\epsilon>0$ there exist $\delta>0$ and
	\[T\in \mathfrak{a}^G_n(U;\epsilon,\delta; \F )\]
	with ${\bf F}(V,|T|)<\epsilon$. 
\end{definition}

\begin{remark}\label{Rem: boundary type}
	Note $T\in \mathcal Z_{n}^G(M;\mathbb{Z}_2)$ is a boundary of some $Q\in{\bf I}^G_{n+1}(M;\mathbb{Z}_2)$. 
	Hence, the varifold in Definition \ref{Def:a.m.} is actually $(G,\mathbb{Z}_2)$-almost minimizing of {\em boundary type} in the sense of \cite[Definition 10]{wang2022min}. 
\end{remark}

\begin{definition}\label{Def:a.m. in regular annuli}
	We say a $G$-varifold $V \in \mathcal{V}^G_n(M)$ is {\it $(G,\mathbb{Z}_2)$-almost minimizing in annuli} if
	for every $p \in M$, there exists $r=r(G\cdot p)>0$ such that $V$ is $(G,\mathbb{Z}_2)$-almost minimizing in $\an (p,s,t)$ for any $0<s<t<r$. 
\end{definition}

In Definition \ref{Def:a.m.}, the flat norm $\F$ is used in $\mathfrak{a}^G_n(U;\epsilon,\delta; {\bf v} ) $ to define the $(G,\mathbb{Z}_2)$-almost minimizing. 
Nevertheless, by the slice representation \cite[Section 2.1.3, ($\Sigma_4$)]{berndt2016submanifolds}, the $G$-annulus $\an(p,s,t)$, $0<s<t<{\rm Inj}(G\cdot p)$, contains no {\em isolated orbit} in the sense of \cite[Definition 2.1]{wang2023min}, and thus it is equivalent to use $\F$, $\mF$, and $\M$ in the above Definition \ref{Def:a.m. in regular annuli} by \cite[Theorem 3.14]{wang2023min}.

\begin{definition}\label{Def: admissible family}
	For any $p\in M$ and $c\geq 1$, let $\mathcal{A}=\{\an (p,s_i,t_i)\}_{i=1}^c$ be a set of $G$-annuli with $0< s_i < t_i < \inj(G\cdot p)$ for all $i=1,\dots,c$. 
	We say $\mathcal{A}$ is a {\em $c$-admissible} family of $G$-annuli, if $0<s_i<t_i<\frac{1}{2}s_{i+1} < \frac{1}{2}t_{i+1}$, for all $i=1,\dots,c-1$. 
\end{definition}

\begin{lemma}\label{Lem: admissible to all annuli}
	Let $V\in\V^G_n(M)$ be a $G$-varifold, $c\geq 1$ be an integer. 
	If for any $c$-admissible family of $G$-annuli $\mathcal{A}$, $V$ is $(G,\mathbb{Z}_2)$-almost minimizing in some $\an\in\mathcal{A}$. 
	Then $V$ is $(G,\mathbb{Z}_2)$-almost minimizing in annuli. 
\end{lemma}
\begin{proof}
	This lemma follows immediately from a contradiction argument. 
\end{proof}

By a combinatorial argument of Pitts \cite{pitts2014existence}, we have the following existence theorem. 
\begin{theorem}\label{Thm: exist amv}
	Let $X$ be an $m$-dimensional cubical complex, $S=\{\varphi_i\}_{i\in\N}$ be a sequence of maps 
	$\varphi_i: X( n_i)_{0}\rightarrow \Z_{n}^G(M ;\mZ_2) $ so that $\lim_{i\to\infty} n_i=\infty$, $\lim_{i\to\infty} \mathbf{f}(\varphi_i)=0$, and every $V\in \mathbf{C}(S)$ is stationary in $M$. 
	Suppose for any element $V\in \mathbf{C}(S)$, there is a $(3^m)^{3^m}$-admissible family of $G$-annuli $\mathcal{A}_V$ so that $V$ is not $(G,\mathbb{Z}_2)$-almost minimizing in every $\an\in\mathcal{A}_V$. 
	Then there is another sequence of maps $S^*=\{\varphi_i^*\}_{i\in\N}$, 
	\[\varphi_i^*: X(n_i+l_i)_{0}\rightarrow \Z_{n}^G(M ;\mZ_2)\]
	for some $l_i\in\N$, such that
	\begin{itemize}
		\item $\varphi_i$ and $\varphi_i^*$ are homotopic in $\Z_{n}^G(M; \M ;\mZ_2)$ with fineness tending to zero;
		\item $\mathbf{L}(S^*)< \mathbf{L}(S)$.
	\end{itemize}
\end{theorem}
\begin{proof}
	The proof follows generally the arguments in \cite[Theorem 4.10]{pitts2014existence}. 
	
	Let $c=(3^m)^{3^m}$. 
	Given any $V\in \mathbf{C}(S)$, let $\mathcal{A}_V$ be the $c$-admissible family of $G$-annuli so that $V$ is not $(G,\mathbb{Z}_2)$-almost minimizing in every $\an\in\mathcal{A}_V$. 
	We can write $\mathcal{A}_V = \big\{ a_i(V):=\an(p, r_i-s_i, r_i+s_i)  \big\}_{i=1}^c$, where $\{r_i\}_{i=1}^c$ and $\{s_i\}_{i=1}^c$ are positive numbers with $0<r_1-s_1$, $2(r_i + s_i) < r_{i+1} - s_{i+1}$, $1\leq i\leq c-1$. 
	Hence, there exists a positive number $s(V) > 0$ sufficiently small so that $\big\{ A_i(V):=\an(p, r_i-s_i-s(V), r_i+s_i+s(V)) \big\}_{i=1}^c$ is also a $c$-admissible family of $G$-annuli. 
	This is parallel to \cite[Theorem 4.10, Part 1]{pitts2014existence}. 
	
	By \cite[Theorem 3.14]{wang2023min}, the construction in \cite[Theorem 4.10, Part 2]{pitts2014existence} can be carried out similarly in our equivariant setting. 
	Additionally, the arguments in \cite[Theorem  4.10, Part 3-8]{pitts2014existence} would carry over with Lemma \ref{Lem: isoperimetric} in place of \cite[1.14]{almgren1962homotopy}. 
	We shall also use the constant $s(V)>0$ constructed as above in \cite[Theorem  4.10, Part 3]{pitts2014existence}. 
	In \cite[Theorem  4.10, Part 9]{pitts2014existence}, we can use the $G$-invariant distance function $\dist_M(G\cdot p,\cdot)$ in place of $u$. 
	Then the rest parts of \cite[Theorem  4.10]{pitts2014existence} are purely combinatorial which would carry over in the equivariant setting. 
\end{proof}

Suppose $\{\Phi_i\}\subset {\bm \Pi}$ and $S=\{\varphi_i\}_{i\in\N}$ are the tightened sequences of mappings given by Proposition \ref{Prop:pulltight} and (\ref{Eq: discrete pull-tight}). 
Then we can apply the arguments in \cite[Theorem 6]{wang2022min} to $S$ with Theorem \ref{Thm: exist amv} in place of \cite[Theorem 5]{wang2022min} to obtain the following corollary:
\begin{corollary}\label{Cor: exist stationary and a.m. varifolds}
	Let $X$ be an $m$-dimensional cubical complex and ${\bf \Pi}\in \big[ X, \Z_{n}^G(M;\mF;\mZ_2)\big]$ be a continuous $G$-homotopy class. 
	Then for any min-max sequence $\{\Phi_i\}_{i\in\N}$ in ${\bf \Pi}$, there is a $G$-varifold $V\in {\bf C}(\{\Phi_i\}_{i\in\N})$ so that
	\begin{itemize}
		\item[(i)] $V$ is stationary in $M$;
		\item[(ii)] for any $(3^m)^{3^m}$-admissible family of $G$-annuli $\mathcal{A}$, $V$ is $(G,\mathbb{Z}_2)$-almost minimizing in some $\an\in\mathcal{A}$, which implies $V$ is $(G,\mathbb{Z}_2)$-almost minimizing in annuli by Lemma \ref{Lem: admissible to all annuli}. 
	\end{itemize}
\end{corollary}

\subsection{Regularity results}

\begin{definition}\label{Def: good replacements property}
	Let $U\subset M$ be an open $G$-set and $V\in\V^G_n(M)$ be a $G$-varifold that is stationary in $U$. 
	Given a compact $G$-set $K\subset U$ and a $G$-varifold $V^*\in \V^G_n(M)$, we say $V^*$ is a {\em $G$-replacement} of $V$ in $K$, if 
	\begin{itemize}
		\item[(i)]  $V \llcorner (M\setminus K) = V^* \llcorner (M\setminus K)$;
		\item[(ii)] $\|V\|(M) = \|V^*\|(M)$;
		\item[(iii)] $V^*$ is stationary in $U$;
		\item[(iv)] $V^*\llcorner \interior (K)$ is integer rectifiable, and $\Sigma := {\rm spt}(\|V^*\|)\cap \interior(K) $ is a $G$-invariant smoothly embedded stable minimal hypersurface.
	\end{itemize}
	We say $V$ has {\em good $G$-replacement property} in $U$, if for any finite sequence of compact $G$-sets $\{K_i\subset U\}_{i=1}^q$, $V^{(i)}$ has a $G$-replacement $V^{(i+1)}$ in $K_{i+1}$, where $V^{(0)} = V$, $i=0,\dots, q-1 $. 
	Moreover, if for any $p\in M$, there exists $r_{g.r.}(G\cdot p) > 0$ so that $V$ has good $G$-replacement property in every $\an(p,s,t)$, $0<s<t<r_{g.r.}(G\cdot p)$, then we say $V$ has {\em good $G$-replacement property in annuli}. 
\end{definition}

\begin{remark}\label{Rem: good replacement property}
	Suppose $V$ is stationary in $M$ and has good $G$-replacement property in annuli. 
	Let $V^*$ be a $G$-replacement of $V$ in $\Clos(\an(p,s,t))$, where $0<s<t<r=r_{g.r.}(G\cdot p)$, $r_{g.r.}:M/G\to\R^+$ is given in Definition \ref{Def: good replacements property}. 
	Then it follows from Definition \ref{Def: good replacements property}(i)(iii) that $V^*$ is also stationary in $M$. 
	Meanwhile, by shrinking $r_{g.r.}$ for $G\cdot q\neq G\cdot p$, we see that $V^*$ also has good $G$-replacement property in annuli. 
\end{remark}


\begin{proposition}\label{Prop: a.m. implies good replacements} 
	If $3\leq {\rm codim}(G\cdot p) \leq 7$ for all $p\in M$, and $V\in\V^G_n(M)$ is $(G,\mathbb{Z}_2)$-almost minimizing in an open $G$-set $U\subset M$. 
	Then $V$ has good $G$-replacement property in $U$. 
\end{proposition}
\begin{proof}
	Note that the singular set of a $G$-hypersurface is also $G$-invariant. 
	Hence, using the assumption $3\leq {\rm codim}(G\cdot p) \leq 7,\forall p\in M$, the proposition follows from the good $G$-replacement results in \cite[Proposition 2, 3]{wang2022min} and the optimal regularity for stable minimal hypersurfaces \cite{schoen1981regularity}\cite{wickramasekera2014general}. 
\end{proof}

By Proposition \ref{Prop: a.m. implies good replacements}, the $G$-varifold $V$ in Corollary \ref{Cor: exist stationary and a.m. varifolds} has good $G$-replacement property in annuli. 
Additionally, one should notice that the $G$-varifold $V$ in Corollary \ref{Cor: exist stationary and a.m. varifolds} satisfies a slightly stronger property to some extent by the following lemma:
\begin{lemma}\label{Lem: replace in all annuli}
	Let $V\in\V^G_n(M)$ be a $G$-varifold and $c\geq 1$ be an integer. 
	If for any $c$-admissible family of $G$-annuli $\mathcal{A}$, $V$ has good $G$-replacement property in some $\an\in\mathcal{A}$. 
	Then $V$ has good $G$-replacement property in $G$-annuli. 
\end{lemma}
\begin{proof}
	This is parallel to Lemma \ref{Lem: admissible to all annuli} and follows from a contradiction argument. 
\end{proof}

The following regularity results come from the arguments in \cite[Section 6]{wang2022min}. 
\begin{theorem}[Regularity Theorem]\label{Thm: regularity of varifolds with good replacements}
	Suppose $3\leq {\rm codim}(G\cdot p)\leq 7$ for all $p\in M$. 
	If $c\geq 1$ is an integer and $V \in \mathcal{V}^G_n(M)$ satisfies: 
	\begin{itemize}
		\item $V$ stationary in $M$,
		\item for any $c$-admissible family of $G$-annuli $\mathcal{A}$, $V$ has good $G$-replacement property in some $\an\in\mathcal{A}$. 
	\end{itemize}
	Then $V$ is an integral varifold induced by a closed, smooth, embedded, $G$-invariant minimal hypersurface. 
	In particular, the regularity result also holds if $V\in \V_n^G(M)$ is stationary in $M$ and has good $G$-replacement property in annuli. 
\end{theorem}
\begin{proof}
	The proof is essentially the same as \cite[Theorem 7]{wang2022min} and we only point out a few modifications. 
	Firstly, note the good $G$-replacement property holds in every small $G$-annuli instead of only regular $G$-annuli (\cite[(12)]{wang2022min}). 
	Therefore, the assumption on $M\setminus M^{reg}$ in \cite[Theorem 7]{wang2022min} is now redundant. 
	Additionally, using Proposition \ref{Prop: a.m. implies good replacements} in place of \cite[Proposition 4]{wang2022min}, we see \cite[Lemma 10, 11]{wang2022min} are also valid at every $p\in\spt(\|V\|)$. 
	Then by the splitting property \cite[Lemma 11(iii)]{wang2022min}, one can verify that \cite[Proposition 5]{wang2022min} also holds provided $3\leq {\rm codim}(G\cdot p) \leq 7$ for all $p\in M$. 
	Moreover, we shall also use the assumption $n-6\leq\dim(G\cdot p) \leq n-2$ to remove the singularity for stable minimal $G$-hypersurfaces by \cite{schoen1981regularity}\cite[Corollary 2]{wickramasekera2014general}. 
	%
\end{proof}

Note if one allows a singular set of Hausdorff dimension no more than $n-7$, then Theorem \ref{Thm: regularity of varifolds with good replacements} also holds for all $n\geq 2$ with `hypersurface' replaced by `generalized hypersurface' in the sense of \cite[Definition 2.1]{liu2021existence}. 

At the end of this subsection, we show that the $G$-replacement coincides with the original varifold $V$ if $V$ is smooth embedded. 
\begin{proposition}\label{Prop: replacement equals to origin}
	Suppose $3\leq {\rm codim}(G\cdot p)\leq 7$ for all $p\in M$. 
	Take any $p\in M$ and $0<r<\frac{\inj(G\cdot p)}{2}$ so that $B^G_{t}(p)$ has mean convex boundary for $0<t<2r$. 
	Suppose $V\in \V_n^G(M)$ has good $G$-replacement property in $B^G_{2r}(p)$, and $\spt(\|V\|)\cap B^G_{2r}(p)$ is a smooth embedded minimal hypersurface without boundary in $B^G_{2r}(p)$. 
	Let $V^*$ be a $G$-replacement of $V$ in $\Clos(B^G_s(p))$ for some $s\in (\frac{r}{2}, r )$. 
	If $\spt(\|V\|)$ is transversal to $\partial B^G_s(p)$ or $\spt(\|V\|)\cap \Clos(B^G_{s}(p)) = \emptyset$, then $V=V^*$ and $V$ is stable in $B^G_s(p)$. 
\end{proposition}
\begin{proof}
	By the maximum principle and the fact that $V=V^*$ in $M\setminus\Clos(B^G_s(p))$, we can assume without loss of generality that $\spt(\|V\|)\cap B^G_{2r}(p)$ has connected components $\{\Sigma_k\}_{k=1}^K$ so that $\partial \Sigma_k \subset \partial B^G_{2r}(p)$, $\Sigma_k\cap B^G_{s}(p)\neq\emptyset$, and $\Sigma_k$ intersects with $\partial B^G_{s}(p) $ transversally, $k=1,\dots,K$. 
	By the smooth gluing results (c.f. [28, (7.18)-(7.38)]) used in \cite[Theorem 7]{wang2022min} with $V,V^*$, in place of $V^*,V^{**}$, we have $V\llcorner \an(p,\frac{r}{2}, s) = V^*\an(p,\frac{r}{2}, s )$. 
	Now, it follows from the classical unique continuation result \cite[Theorem 5.3]{de2013existence} that $\spt(\|V\|)\cap B^G_s(p) \subset \spt(\|V^*\|)\cap B^G_s(p)$. 
	Since the transversality implies $\|V\|(\partial B^G_s(p)) = 0$, one easily verifies that $ \|V^*\|(\partial B^G_s(p)) = 0$ and $\|V^*\|(B^G_s(p)) = \|V\|(B^G_s(p))$. Thus, $V=V^*$ indicating the stability in $B^G_s(p)$. 
\end{proof}

\subsection{Equivariant min-max theorem}
At the end of this section, we state our equivariant min-max theorem as an improved version of \cite[Theorem 8]{wang2022min}. 
For simplicity, let us first introduce the following concept. 
\begin{definition}\label{Def: amv in admissible annuli}
	Given $c\in\mZ_+$, $V\in\V_n^G(M)$ is said to satisfy the {\em $(\ast_c^G)$-property}, if for any $c$-admissible family of $G$-annuli $\mathcal{A}$, $V$ has good $G$-replacement property in some $\an\in\mathcal{A}$. 
\end{definition}

\begin{theorem}\label{Thm: min-max}
	Let $(M^{n+1}, g_{_M})$ be a closed Riemannian manifold with a compact Lie group $G$ acting as isometries. 
	Suppose $3\leq {\rm codim}(G\cdot p)\leq 7$ for all $p\in M$. 
	For any $G$-homotopy class ${\bm \Pi}\in\big[X^m,\Z_n^G(M;\mF;\mZ_2)\big]$, let $c=(3^m)^{3^m}$, $\{\Phi_i\}_{i\in\N}$ be a min-max sequence of ${\bm \Pi}$. 
	Then there exists a stationary $G$-varifold $V\in\mathbf{C}(\{\Phi_i \}_{i\in\N})$ with $(\ast_c^G)$-property, 
	and thus $V$ is an integral $G$-varifold induced by a closed smooth embedded $G$-invariant minimal hypersurface. 
\end{theorem}
\begin{proof}
	By Corollary \ref{Cor: exist stationary and a.m. varifolds}, there is a stationary $G$-varifold $V\in\mathbf{C}(\{\Phi_i \}_{i\in\N})$ so that for any $(3^m)^{3^m}$-admissible family of $G$-annuli $\mathcal{A}$, $V$ is $(G,\mathbb{Z}_2)$-almost minimizing in some $\an\in\mathcal{A}$, which implies $V$ has good $G$-replacement property in this $\an$ by Proposition \ref{Prop: a.m. implies good replacements}. 
	Therefore, this $G$-varifold $V$ satisfies the $(\ast_c^G)$-property, and the regularity of $V$ then follows directly from Theorem \ref{Thm: regularity of varifolds with good replacements}. 
\end{proof}


\section{Min-max $G$-hypersurfaces and Compactness Theorem}\label{Sec: compactness}

The main purpose of this section is to show a compactness theorem together with a generic finite result for min-max $G$-hypersurfaces. 
To begin with, we make the following definitions. 

\begin{definition}\label{Def: min-max G-hypersurfaces}
	For any $c\in\mZ_+$, if $\Sigma \in \V_n^G(M)$ is a stationary $G$-varifold satisfying the $(\ast_c^G)$-property (Definition \ref{Def: amv in admissible annuli}), then we say $\Sigma$ is a {\em min-max $(c,G)$-hypersurface (with multiplicity)}. 
	Moreover, if $\Sigma $ is constructed by Theorem \ref{Thm: min-max} from a $G$-homotopy class ${\bm \Pi}$, then we say $\Sigma$ is a min-max $(c,G)$-hypersurface {\em corresponding to ${\bm \Pi}$}. 
	
	In general, we say $\Sigma$ is a {\em min-max $G$-hypersurface (with multiplicity)} if it is a min-max $(c,G)$-hypersurface for some $c\in\mZ_+$. 
\end{definition}

\begin{definition}\label{Def: G-connected}
	A $G$-hypersurface $\Sigma$ is said to be {\em $G$-connected} if for any two connected components $\Sigma_1,\Sigma_2$ of $\Sigma$, there exists $g\in G$ so that $g\cdot \Sigma_1 = \Sigma_2$. 
\end{definition}

By definitions, if $c_1>c_2\geq 1$ and $\Sigma$ is a min-max $(c_2,G)$-hypersurface, then $\Sigma$ is also a min-max $(c_1,G)$-hypersurface. 
Additionally, by the regularity theorem \ref{Thm: regularity of varifolds with good replacements}, it is reasonable to call the stationary $G$-varifold satisfying the $(\ast_c^G)$-property a $G$-hypersurface with multiplicity. 
Hence, we can define the $G$-index of a min-max $(c,G)$-hypersurface $\Sigma$ as ${\rm Index}_G(\Sigma) : = {\rm Index}_G(\spt(\|\Sigma\|))$. 
Moreover, we have the following compactness theorem for min-max $(c,G)$-hypersurfaces.

\begin{theorem}[Compactness Theorem]\label{Thm: compactness theorem}
	Let $(M^{n+1},g_{_M})$ be a closed Riemannian manifold and $G$ be a compact Lie group acting by isometries on $M$ so that $3\leq {\rm codim}(G\cdot p) \leq 7$ for all $p\in M$. 
	Suppose $c\in\mZ_+$, $\{\Sigma_i \}_{i\in\N}$ is a sequence of min-max $(c,G)$-hypersurfaces (with multiplicities) with $\sup_{i\in\N} \|\Sigma_i\|(M) \leq C <+\infty$. 
	Then $\Sigma_i $ converges (up to a subsequence) in the varifold sense to an integral stationary $G$-varifold $\Sigma_\infty\in\V_n^G(M)$ so that $\Sigma_\infty$ is also a min-max $(c,G)$-hypersurface (with multiplicity). 
	Moreover, there exists a finite union of orbits $\mathcal{Y}=\cup_{k=1}^K G\cdot p_k$ such that the convergence (up to a subsequence) is locally smooth and graphical on $\spt(\|\Sigma_\infty\|)\setminus \mathcal{Y} $ with a constant integer multiplicity on each $G$-connected component of $\spt(\|\Sigma_\infty\|)$. 
\end{theorem}
Note the assumption $3\leq {\rm codim}(G\cdot p) \leq 7, \forall p\in M$, ensures the regularity of stable $G$-hypersurfaces by \cite{schoen1981regularity}\cite{wickramasekera2014general}. 

\begin{proof}
	By the compactness theorem of varifolds, we have $\Sigma_i$ converges (up to a subsequence) to a stationary $G$-varifold $\Sigma_\infty$. 
	Then for any $c$-admissible family of $G$-annuli $\mathcal{A}$, each $\Sigma_{i}$ has good $G$-replacement property in some $A_{(i)}\in\mathcal{A}$. 
	Since $\mathcal{A}$ contains only a finite number of $G$-annuli, there exist $A\in\mathcal{A}$ and a (not relabeled) subsequence $\{\Sigma_{i}\}_{i\in\N}$ so that every $\Sigma_{i}$ has good $G$-replacement property in $A$. 
	We are now going to show that $\Sigma_\infty$ has good $G$-replacement property in such $A\in\mathcal{A}$, which implies $\Sigma_\infty$ satisfies the $(\ast_c^G)$-property and is a min-max $(c,G)$-hypersurface. 
	
	Indeed, given any compact $G$-set $K\subset A$, let $\Sigma_{i}^{*}$ be a $G$-replacement of $\Sigma_{i}$ in $K$. 
	Hence, by Definition \ref{Def: good replacements property}(i)(ii)(iii), we have $\Sigma_{i}^*$ converges (up to a subsequence) to a stationary $G$-varifold $\Sigma_\infty^*$. 
	Additionally, by Definition \ref{Def: good replacements property}(iv) and the compactness theorem for stable minimal hypersurfaces (\cite[Corollary 1]{schoen1981regularity}), we also have that $\Sigma_\infty^*$ is a $G$-replacement of $\Sigma_\infty$ in $K$. 
	By repeating this procedure, we see that $\Sigma_\infty$ has good $G$-replacement property in the $G$-annulus $A$. 
	
	To show the convergence is locally smooth and graphical, we first show the above subsequence $\{\Sigma_i\}_{i\in\N}$ can be taken uniformly around an orbit:
	\begin{claim}\label{Claim: uniform subsequence}
		For any $p\in \spt(\|\Sigma_\infty\|)$, there exists a positive real number $r(G\cdot p)>0$ and a (not relabeled) subsequence $\{\Sigma_{i}\}_{i\in\N}$ so that if $0<s<t\leq r(G\cdot p)$, then for $i$ sufficiently large, each $\Sigma_i$ has good $G$-replacement property in $\an(p,s,t)$. 
	\end{claim}
	\begin{proof}[Proof of Claim \ref{Claim: uniform subsequence}]
		For any $p\in \spt(\|\Sigma_\infty\|)$, let $\mathcal{A}^{(0)} := \{A^{(0)}_k = \an(p, s_k^{(0)}, t_k^{(0)})\}_{k=1}^c$ be a $c$-admissible family of $G$-annuli, where $0<s_k^{(0)}<t_k^{(0)}<\frac{1}{2}s_{k+1}^{(0)} < \frac{1}{2}t_{k+1}^{(0)}$. 
		By the above constructions, there is at least one $A_k^{(0)}$ with infinitely many $\Sigma_i$ having good $G$-replacement property in it. 
		We can then define
		\[k_0 := \max  \left\{ k\in\{1,\dots,c\} : \begin{array}{c} \mbox{there is a subsequence $\{\Sigma_i^k\}_{i\in\N}$ having}\\ \mbox{good $G$-replacement property in } A^{(0)}_k \end{array} \right\},\]
		$\{\Sigma^{(0)}_i := \Sigma^{k_0}_i\}_{i\in\N}$, and $\mathcal{A}^{(1)} := \{A^{(1)}_k = \an(p, s_k^{(1)}, t_k^{(1)})\}_{k=1}^c $, where $\{\Sigma^{k_0}_i\}_{i\in\N}$ is the subsequence given by the choice of $k_0$, and
		\[
		s_k^{(1)} = \left\{\begin{array}{ll}{ \frac{1}{2}s_k^{(0)} ,} & {k\leq k_0} \\ {s_k^{(0)}, }&{k > k_0}\end{array}\right. ,   \qquad t_k^{(1)} = \left\{\begin{array}{ll}{ \frac{1}{2}t_k^{(0)} ,} & {k<k_0} \\ {t_k^{(0)}, } & {k\geq k_0} \end{array}\right.  .
		\]
		Clearly, $\mathcal{A}^{(1)}$ is also a $c$-admissible family of $G$-annuli. 
		Repeating this procedure gives a sequence of numbers $k_j\in \{1,\dots, c \}$, subsequences $\{\Sigma^{(j)}_i\}_{i\in\N}$, and a sequence of $c$-admissible family $\mathcal{A}^{(j)} := \{ A^{(j)}_{k} = \an(p, s_k^{(j)}, t_k^{(j)}) \}_{k=1}^c$, $j=0,1,\dots$, so that 
		\begin{itemize}
			\item $\{\Sigma^{(j+1)}_i\}_{i\in\N}$ is a subsequence of $\{\Sigma^{(j)}_i\}_{i\in\N}$;
			\item $\{\Sigma^{(j)}_i\}_{i\in\N}$ all have good $G$-replacement property in $A^{(j)}_{k_j}$;
			\item for $k> k_j$, only a finite number of $\{\Sigma^{(j)}_i\}_{i\in\N}$ have good $G$-replacement property in $A^{(j)}_{k}$; 
			\item $A^{(j)}_{k} = A^{(j+1)}_k$ for $k > k_{j}$. 
		\end{itemize} 
		By the last three bullets, we have $ k_{j+1} \leq k_j$ for all $j\geq 1$. 
		Hence, after a finite number of times, say $J$ times, we have $k_j \equiv k_J$ for all $j\geq J$. 
		Now, take a diagonal subsequence $\Sigma_i := \Sigma_{J+i}^{(J+i)}$ and set $r(G\cdot p) := t^{(J)}_{k_J} $. 
		Then the claim follows from the first two bullets and the fact that $\lim_{j\to\infty}s^{(j)}_{k_j} = \lim_{j\to\infty}\frac{s^{(J)}_{k_J}}{2^{j-J}} = 0 $, $t^{(j)}_{k_j}=t^{(J)}_{k_J}$, $\forall j\geq J$.
	\end{proof}
	
	For simplicity, we abbreviate `good $G$-replacement property' to `GGRP' in the following.  Now, for any subsequence $\{\widetilde{\Sigma}_i \}_{i\in\N}\subset \{\Sigma_i\}_{i\in\N}$, we define the set $\mathcal{Y}(\{\widetilde{\Sigma}_i \}_{i\in\N})$ of bad convergence points on $\spt(\|\Sigma_\infty\|)$ by 
	\[ 
	\mathcal{Y}(\{\widetilde{\Sigma}_i \}) := \{p\in \spt(\|\Sigma_\infty\|): \forall \epsilon>0, N\geq 1, \exists i\geq N, \mbox{ s.t. } \widetilde{\Sigma}_i  \mbox{ is unstable in } B_\epsilon(p)\}, 
	\]
	which is clearly a $G$-invariant subset of $\spt(\|\Sigma_\infty\|)$. 
	By Proposition \ref{Prop: replacement equals to origin}, we see 
	\begin{equation}
		\mathcal{Y}(\{\widetilde{\Sigma}_i \}) \subset \mathcal{Z}(\{\widetilde{\Sigma}_i\}) 
		: =\left\{\begin{array}{l|l} p\in \spt(\|\Sigma_\infty\|)   & \begin{array}{l} \forall \epsilon>0, N\geq 1, \exists i\geq N \mbox{ so that }\widetilde{\Sigma}_i \\ \mbox{does not have GGRP in $B_\epsilon^G(p)$} \end{array}\end{array}\right\}. \nonumber
	\end{equation}
	By the compactness theorem for stable minimal hypersurfaces in \cite{schoen1981regularity}, it remains to show the existence of a subsequence $\{\widetilde{\Sigma}_i \}_{i\in\N}$ so that $\mathcal{Z}(\{\widetilde{\Sigma}_i\})$ contains only a finite number of orbits. 
	We next construct such a subsequence inductively. 
	
	If $\mathcal{Z}(\{\Sigma_i\}) = \emptyset$, then we are done. 
	Suppose $G\cdot p_1\subset \mathcal{Z}(\{\Sigma_i\})$. 
	Then by definitions, there exists a subsequence $\{\Sigma^{(1)}_i := \Sigma_{j(i)}\}$ so that $j(i)\geq \max\{j(i-1), i\}$, and $\Sigma^{(1)}_i$ does not have GGRP in any $B^G_\epsilon(p_1)$ with $\epsilon\geq \frac{1}{i}$. 
	If $\mathcal{Z}(\{\Sigma_i^{(1)}\}) = G\cdot p_1$, we are done. 
	Otherwise, we take $G\cdot p_2 \subset \mathcal{Z}(\{\Sigma_i^{(1)}\}) \setminus G\cdot p_1 $ and get a subsequence $\{\Sigma^{(2)}_i := \Sigma^{(1)}_{j(i)}\}$ so that $j(i)\geq \max\{j(i-1), i\}$ and $\Sigma^{(2)}_i$ does not have GGRP in any $B^G_\epsilon(p_1)$ or $B^G_\epsilon(p_2)$ with $\epsilon\geq \frac{1}{i}$. 
	Then we compare $\mathcal{Z}(\{\Sigma_i^{(2)}\}) $ with $\cup_{k=1}^2 G\cdot p_k $ and repeat. 
	\begin{claim}\label{Claim: finite steps}
		The above procedure will stop after a finite number of times. 
	\end{claim}
	\begin{proof}[Proof of Claim \ref{Claim: finite steps}]
		Suppose the claim is false. Then we have a sequence of different orbits $\{G\cdot p_k\}_{k\in\N}$ and subsequences $\{\Sigma^{(k)}_i \}_{i\in\N}$, $k\in\N$, so that each $\Sigma^{(k)}_i$ does not have GGRP in any $B^G_\epsilon(p_j)$, where $\epsilon\geq \frac{1}{i}$ and $j\leq k$. 
		By the compactness of $\spt(\|\Sigma_\infty \|)$, we may assume $p_k \to p \in \spt(\|\Sigma_\infty \|)$ as $k\to\infty$. 
		Let $\epsilon_k := \frac{1}{2}\dist_M(G\cdot p_k, G\cdot p)$ for each $k\in\N$. 
		Note we can always assume $\epsilon_k>0$ because $\{G\cdot p_k\}_{k\in\N}$ are pairwise different. 
		By taking positive integers $i_1<i_2<...$ with $\epsilon_k\geq \frac{1}{i_k}$, we obtain a diagonal subsequence $\Sigma^{(\infty)}_k := \Sigma^{(k)}_{i_k}$ so that for any $j\leq k$, $\Sigma^{(\infty)}_k$ does not have GGRP in $B^G_{\epsilon_j}(p_j)$. 
		
		Next, we apply Claim \ref{Claim: uniform subsequence} to $\{\Sigma^{(\infty)}_k\}_{k\in\N}$ and $G\cdot p$, which gives a  (not relabeled) subsequence and a constant $r>0$ so that for any $j\in\N$, $\Sigma^{(\infty)}_k$ has GGRP in $\an(p, \frac{d_j}{4} , r)$ for $k$ large enough, where $d_j = \dist_M(G\cdot p_j, G\cdot p) > 0$. 
		On the other hand, we have $B^G_{\epsilon_j}(p_j)\subset \an(p, \frac{d_j}{4} , r)$ for $j$ large enough, and thus $\Sigma^{(\infty)}_k$ has GGRP in $B^G_{\epsilon_j}(p_j)$ for $k\geq j$ large enough, which contradicts the choice of $\Sigma^{(\infty)}_k$. 
	\end{proof}
	
	Therefore, there are $K\in\N$ and subsequence $\{ \widetilde{\Sigma}_i := \Sigma^{(K)}_i  \}_{i\in\N} $ so that $\mathcal{Z}(\{\widetilde{\Sigma}_i\}) =\cup_{k=1}^K G\cdot p_k $, which gives the locally smooth convergence except at $\mathcal{Y}:=\cup_{k=1}^K G\cdot p_k$. 
\end{proof}

\begin{remark}
	It seems Claim \ref{Claim: uniform subsequence} does not directly give the smooth convergence result in Theorem \ref{Thm: compactness theorem}. 
	Indeed, for any $p_1\in \spt(\|\Sigma_\infty\|)$, it follows from Claim \ref{Claim: uniform subsequence}, Proposition \ref{Prop: replacement equals to origin}, and \cite[Theorem 2]{schoen1981regularity}, that after passing to a subsequence $\{\Sigma^{(1)}_i\}\subset \{\Sigma_i\}$, the convergence is locally smooth and graphical in $\an(p_1,s,r_1)$ for $0<s<r_1$ small enough. 
	However, for $p_2\in \spt(\|\Sigma_\infty\|) \setminus B^G_{r_1}(p_1)$, the subsequence $\{\Sigma^{(2)}_i\}\subset \{\Sigma_i\}$ given by Claim \ref{Claim: uniform subsequence} may be pairwise different to $\{\Sigma^{(1)}_i\}$. 
	Besides, if we take $\{\Sigma^{(2)}_i\}\subset \{\Sigma^{(1)}_i\}$, then the radius $r_2$ associated to $p_2$ in Claim \ref{Claim: uniform subsequence} may depend on the choice of $\{\Sigma^{(1)}_i\}$, and it is not clear whether an inductive argument would stop after finite steps. 
\end{remark}

\begin{remark}\label{Rem: various metrics}
	Note the compactness theorem for stable minimal hypersurfaces (\cite{schoen1981regularity}) also holds for varying Riemannian metrics (\cite[Theorem A.6]{sharp2017compactness}). 
	Hence, let $\{g_{_M}^i\}_{i\in \N}$ be a sequence of Riemannian metrics with $\lim_{k\to\infty}g_{_M}^i=g_{_M}$ in the smooth topology. 
	If the min-max $(c,G)$-hypersurface $\Sigma_i$ in Theorem \ref{Thm: compactness theorem} is $g_{_M}^i$-stationary and satisfies the $(\ast_c^G)$-property under the metric $g_{_M}^i$. 
	Then $\Sigma_i $ converges (up to a subsequence) in the varifold sense to a min-max $(c,G)$-hypersurface $\Sigma$ in $(M,g_{_M})$. 
	Additionally, the convergence is also locally smooth and graphical except for a finite number of orbits. 
\end{remark}

Suppose $\Sigma$ is a $G$-invariant hypersurface without $G$-invariant unit normal vector field. 
Let $S\Sigma:=\{v\in {\bf N}_p\Sigma: p\in\Sigma, |v|=1\}$ be the unit normal bundle of $\Sigma$ and $\exp_{\Sigma}^\perp$ be the normal exponential map. 
Then for $r>0$ sufficiently small, 
\begin{equation}\label{Eq: double cover map}
	E: S\Sigma\times (-r,r) \to B_r(\Sigma), \quad E(v,t):= \exp_{\Sigma}^\perp(t\cdot v),
\end{equation}
is a double cover of a tubular neighborhood of $\Sigma$. 
We can lift the metric $g_{_M}$ to $S\Sigma\times (-r,r)$, and also lift the action of $G$ to $S\Sigma\times (-r,r)$ by $g\cdot (v,t) := (dg(v), t)$, $\forall g\in G$. 
Then since $S\Sigma\times \{0\}\subset \partial (S\Sigma\times (0,r))$, we have 
\begin{equation}\label{Eq: double cover}
	\widetilde{\Sigma} := S\Sigma\times \{0\}
\end{equation}
is a $G$-invariant double cover of $\Sigma$ with a $G$-invariant unit normal. 
In the following theorem, we construct a non-trivial $G$-invariant Jacobi field on the limit hypersurface in Theorem \ref{Thm: compactness theorem}.
Note $\Sigma_1\simeq \Sigma_2$ means they are diffeomorphic. 

\begin{theorem}\label{Thm: Jacobi field}
	Let $(M^{n+1},g_{_M})$ be a closed Riemannian manifold and $G$ be a compact Lie group acting isometrically with $3\leq {\rm codim}(G\cdot p) \leq 7$, $\forall p\in M$. 
	Suppose $\{\Sigma_i \}_{i\in\N}$ and $\Sigma$ are closed smooth embedded $G$-connected $G$-hypersurfaces so that $\Sigma_i\to \Sigma$ as varifolds, and the convergence is locally smooth and graphical on $\Sigma\setminus \mathcal{Y} $ with multiplicity $m\in\N$, where $\mathcal{Y}=\cup_{k=1}^K G\cdot p_k$  is a finite union of orbits. 
	If $\Sigma_i\neq\Sigma$ eventually, then:
	\begin{itemize}
		\item[(1)] If $\Sigma$ has a $G$-invariant unit normal, then $\Sigma$ admits a nontrivial $G$-invariant Jacobi field, and
			\begin{itemize}
				\item[(i)] $m=1$ if and only if $\mathcal{Y}=\emptyset$ and $\Sigma_i \simeq \Sigma$ eventually;
				\item[(ii)] $m\geq 2$ if and only if $\mathcal{Y}\neq\emptyset$ and $\Sigma$ is degenerate stable. 
			\end{itemize}
		\item[(2)] If $\Sigma$ does not have a $G$-invariant unit normal, then its $G$-invariant double cover $\widetilde{\Sigma} $ (\ref{Eq: double cover}) admits a nontrivial $G$-invariant Jacobi field, and
			\begin{itemize}
				\item[(i)] $m=1$ if and only if $\mathcal{Y}=\emptyset$ and $\Sigma_i \simeq \Sigma$ eventually;
				\item[(ii)] $m\geq 2$ implies $\widetilde{\Sigma} $ is degenerate stable and $\Sigma$ is strictly stable. If we further assume $\mathcal{Y}=\emptyset$ in this case, then $m=2$ and $\Sigma_i \simeq \widetilde{\Sigma} $ eventually. 
			\end{itemize}
	\end{itemize}
\end{theorem}

\begin{proof}
	The proof is similar to \cite[Claim 4, 5, 6]{sharp2017compactness} (see also \cite[Appendix A.]{colding2000embedded}), and we list the details here for the sake of completeness. 
	
	{\bf Case 1.} Suppose $\Sigma$ has a $G$-invariant unit normal and $m=1$. 
	
	For any $p\in\Sigma$ and $\epsilon >0$, it follows from the monotonicity formula and $m=1$ that 
	\[ \frac{\|\Sigma\|(B_r(p))}{\omega_n r^n}\leq 1+\epsilon, \quad \|\Sigma\|(\partial B_r(p))=0, \quad\mbox{for some $r=r(p,\epsilon)>0$}.\]
	By the varifolds convergence, we have $\|\Sigma_i\|(B_r(p)) \leq (1+\epsilon)^2\omega_n r^n $ for $i$ sufficiently large. 
	Take $\eta = \eta(\epsilon, r,n )>0$ with $r^n\leq (1+\epsilon)(r-\eta)^n$. 
	Then for any $i$ large enough and $q_i\in\Sigma_i\cap B_\eta(p)$, $\|\Sigma_i\|(B_{r-\eta}(q_i))\leq\|\Sigma_i\|(B_r(p))\leq  (1+\epsilon)^2\omega_n r^n \leq (1+\epsilon)^3\omega_n (r-\eta)^n$, which implies $\frac{\|\Sigma_i\|(B_{\sigma}(q_i))}{\omega_n \sigma^n} \leq (1+\epsilon)^4$ for all $0< \sigma\leq r-\eta$ by the monotonicity formula. 
	It then follows from a special case of Allard's Regularity Theorem (\cite[Theorem 1.1]{white2005local}) that the second fundamental forms $\sup_{q_i\in \Sigma_i\cap B_\eta(p)}|A_{\Sigma_i}(q_i)|$ are uniformly bounded. 
	Hence, the convergence is smooth and graphical at any $p\in\Sigma$ by a standard argument. 
	
	On the other hand, if $\mathcal{Y}=\emptyset$, then since every $\Sigma_i$ is $G$-connected and $\Sigma$ has a $G$-invariant unit normal, we see $\Sigma_i$ can be written as a $G$-invariant graph over $\Sigma$ for $i$ sufficiently large. 
	Hence, we must have $\Sigma_i\simeq \Sigma$ and $m=1$. 
	
	Now we write $\Sigma_i$ as the graph of a $G$-invariant function $u_i:\Sigma\to\mathbb{R}$, i.e. $\Sigma_i = \exp_\Sigma^\perp(u_i\nu)$. 
	Let $h_i := u_i / \|u_i\|_{L^2} $. Then the arguments in {\bf Case 2} imply $h_i$ tends to a non-trivial smooth $G$-invariant function $h$ solving the Jacobi equation $L_\Sigma h = 0$.

	{\bf Case 2.} Suppose $\Sigma$ has a $G$-invariant unit normal and $m\geq 2$. 
	
	By {\bf Case 1}, we have $\mathcal{Y}\neq \emptyset $. 
	It remains to show that $\Sigma$ admits a non-trivial $G$-invariant Jacobi field which does not change the sign. 
	
	Let $\nu$ be the $G$-invariant unit normal of $\Sigma$, which extends to a $G$-invariant vector field $X\in \mathfrak{X}^G(M)$ so that $X= \nabla d_{\pm}$ in a neighborhood of $\Sigma$, where $d_{\pm}$ is the signed distance function to $\Sigma$. 
	Denote by $\{\phi_t\}$ the diffeomorphisms generated by $X$. 
	Fix any compact $G$-invariant domain $\Omega\subset\subset \Sigma\setminus \mathcal{Y}$. 
	Then for $i$ sufficiently large, we have $G$-invariant smooth functions $\{u_i^1<\cdots<u_i^m\}$ on $\Omega$ so that 
	\[\Sigma_i\cap\Omega_\delta = \{\phi_{u_i^1(x)}(x), \dots, \phi_{u_i^m(x)}(x): x\in \Omega \},\]
	where $\delta>0$ and $\Omega_\delta:=\{\phi_t(x):x\in\Omega, t\in (-\delta,\delta )\}$ is a $\delta$-neighborhood around $\Omega$. 
	For $t\in [0,1]$ and $x\in\Omega$, define $v_i(x,t) := (1-t)u_i^1(x) + tu_i^m(x)$ and 
	\[ \Phi_t^i(x) := \phi_{v_i(x,t)}(x), \quad \Sigma_i(t) := \Phi^i_t(\Omega) . \]
	Then we see $\Sigma_i(0),\Sigma_i(1)\subset \Sigma_i$ are two minimal leaves of $\Sigma_i\cap \Omega_\delta$, and $\Sigma_i(t)$ is a smooth family of $G$-hypersurfaces with variation vector field $X_i = \partial \Phi^i_t/\partial t$ given by 
	\[X_i(\Phi_t^i(x)) = \left(u_{i}^{m}(x)-u_{i}^{1}(x)\right) X(\Phi_{t}^{i}(x)) .\]
	
	Next, take any vector field $Z= \eta X\in\mathfrak{X}(M)$ compact supported in $\Omega_\delta$ for some $\eta\in C^\infty_c(\Omega_\delta)$. 
	Consider the diffeomorphisms $\{\Psi_s\}$ generated by $Z$ and the variation
	\[ \Sigma_i(t,s) := \Psi_s(\Sigma_i(t)) = \{\Psi_s(\phi_{v_i(x,t)}(x)): x\in\Omega\}. \]
	By the first variation formula, we have 
	$ \frac{\partial}{\partial s}\big|_{s=0} \mH^n(\Sigma_i(t,s)) = \int_{\Sigma_i(t)} \Div_{\Sigma_i(t)}(Z) d\mH^n .$ 
	Since $\Sigma_i(0),\Sigma_i(1)$ are minimal hypersurfaces, we can follow the computations in \cite[Section 9]{simon1983lectures} (see also \cite[Appendix A]{ambrozio2018compactness}) to show that 
	\begin{eqnarray*}
		 0 &= &\int_{0}^1 \frac{\partial}{\partial t}\frac{\partial}{\partial s}\Big|_{s=0} \mH^n(\Sigma_i(t,s)) dt
		 \\
		 &=& \int_{0}^{1} \int_{\Sigma_{i}(t)} 
		 	\left\langle \nabla^{\perp}(X_{i}^{\perp}), \nabla^{\perp}(Z^{\perp})\right\rangle 
		 	- \Ric_{M}(X_{i}^{\perp}, Z^{\perp})
		 	- |A|^{2}\left\langle X_{i}^{\perp}, Z^{\perp}\right\rangle      d\mH^n dt
		 \\&& +\int_{0}^{1} \int_{\Sigma_{i}(t)} 
		 	\langle X_{i}^{\perp}, H_{i}(t)\rangle \langle Z_{i}^{\perp}, H_{i}(t)\rangle 
		 	- \langle \nabla_{X_{i}^{\perp}}^\perp Z^\perp, H_{i}(t)\rangle 
		 	- \langle [X_{i}^{\perp}, Z^\top ] , H_{i}(t)\rangle               d \mH^{n} dt,
	\end{eqnarray*}
	where $H_{i}(t)$ is the mean curvature vector field of $\Sigma_i(t)$, and $Z\llcorner \partial\Omega_\delta = 0$ is used to cancel the boundary terms. 
	
	Let $\tilde{h}_i := u_i^m - u_i^1$. 
	Using $\Phi^i_t:\Omega \to\Sigma_i(t)$, the calculations on $\Sigma_i(t)$ under the metric $g_{_M}$ can be pulled back to $\Omega$ under the metric $(\Phi^i_t)^*g_{_M}$.
	Since $v_i \to 0  $ uniformly and smoothly as $i\to\infty$, we have $\Phi^i_t\to  id_\Omega$ and $\Sigma_i(t)\to \Omega$ uniformly and smoothly as $i\to\infty$. 
	Therefore, the above equation implies
	\begin{eqnarray}\label{Eq: Jacobi field}
		0 &=& \int_{0}^{1} \int_{\Omega} 
		\nabla \tilde{h}_i \cdot \nabla \eta - \tilde{h}_i\eta\left( \Ric_M(\nu,\nu) + |A|^2 \right) + L_i(t)(\tilde{h}_i,\eta) d\mH^n dt \nonumber
		\\
		&=& \int_{\Omega} 
		\nabla \tilde{h}_i \cdot \nabla \eta - \tilde{h}_i\eta\left( \Ric_M(\nu,\nu) + |A|^2 \right) + L_i(\tilde{h}_i,\eta) d\mH^n,
	\end{eqnarray}
	where $L_i = \int_0^1L_i(t)$, and $L_i(t)(\cdot,\cdot)$ is a first order linear operator	 with coefficients converging to $0$ uniformly and smoothly as $i\to\infty$. 
	
	Fix $p_0\in \Omega$ and take $h_i(p) = \tilde{h}_i(p) / \tilde{h}_i(p_0)$. 
	Since $h_i$ is a $G$-invariant positive solution of an elliptic equation (\ref{Eq: Jacobi field}), the Harnack inequality implies $\sup_{p\in \Omega'} h_i \leq Ch(p_0) =C$, where $p_0\in \Omega'\subset\subset\Omega$ and $C=C(\Omega')$. 
	By the elliptic regularity theory, we have smooth control of $h_i$ on any $\Omega'\subset\subset \Omega$. 
	Therefore, $h_i\to h$ smoothly up to a subsequence. 
	It follows from (\ref{Eq: Jacobi field}) and the maximum principle that $h: M\setminus \mathcal{Y} \to (0,\infty )$ is a positive $G$-invariant solution of
	\[ \triangle_{\Sigma} h + \Ric_M(\nu,\nu)h + |A|^2 h =0.  \]
	If $h$ is uniformly bounded even at $\mathcal{Y}$, then the classical regularity results suggest that $h\nu$ is a smooth $G$-invariant Jacobi field on $\Sigma$ which does not change the sign by the maximum principle. 
	Hence, (2)(ii) is done provided $h$ is bounded at $\mathcal{Y}$.
	
	To show $h$ is uniformly bounded at $G\cdot p \subset \mathcal{Y}$, let $\B_{r}^\Sigma(G\cdot p)$ and $\C_r(G\cdot p)$ be given in Proposition \ref{Prop: foliation}. 
	Take two $G$-invariant functions $w_i^1,w_i^m$ on $\B_{r}^\Sigma(G\cdot p)$ so that $w_i^j\llcorner \partial \B_{r}^\Sigma(G\cdot p)) = u_i^j$ and $\|w_i^j\|_{C^{2,\alpha}(\B_{r}^\Sigma(G\cdot p))} \leq C_0 \|u_i^j\|_{C^{2,\alpha}(\partial \B_{r}^\Sigma(G\cdot p)))}$, where $j=1,m$. 
	For $i$ large enough, $\|u_i^j\|_{C^{2,\alpha}(\partial \B_{r}^\Sigma(G\cdot p)))}$ is sufficiently small, Proposition \ref{Prop: foliation} gives a foliation of $\C_r(G\cdot p)$ formed by $G$-invariant minimal graphs $v^j_{i,t}$ over $\B_{r}^\Sigma(G\cdot p) $ so that 
	\[ v^j_{i,t} \llcorner \partial \B_{r}^\Sigma(G\cdot p) = t + w^j_i = t+ u^j_i,\quad t\in [-1, 1],~j\in\{1,m\}. \]
	By the maximum principle, $\Sigma_i \cap \C_r(G\cdot p)$ lies under the graph $v^m_{i,0}$ and also lies above the graph $v^1_{i,0}$. 
	It then follows from Proposition \ref{Prop: foliation}(iv) that 
	\[ \sup_{{\rm dmn}(\tilde{h}_i)} \tilde{h}_i \leq \sup_{\B_{r}^\Sigma(G\cdot p)}|v^m_{i,0} - v^1_{i,0}| \leq 2\sup_{\partial \B_{r}^\Sigma(G\cdot p)}|v^m_{i,0} - v^1_{i,0}| = 2\sup_{\partial\B_{r}^\Sigma(G\cdot p)}\tilde{h}_i,\]
	which implies $\tilde{h}_i$ as well as $h$ is uniformly bounded. 
	
	{\bf Case 3.} Suppose $\Sigma$ does not have a $G$-invariant unit normal. 
	
	Using $E$ in (\ref{Eq: double cover map}), we consider the double covers $\widetilde{\Sigma}_i := E^{-1}(\Sigma_i)$ and $\widetilde{ \Sigma} := E^{-1}(\Sigma)$. 
	If $m=1$, then we can use the arguments in {\bf Case 1} to deliver (2)(i). 
	If $m\geq 2$, then $\widetilde{\Sigma}_i \to \widetilde{ \Sigma}$ locally smoothly and graphically on $\widetilde{ \Sigma}\setminus E^{-1}(\mathcal{Y}) $ with multiplicity $m$. 
	By {\bf Case 2}, $\widetilde{\Sigma}$ has a non-trivial $G$-invariant Jacobi field $\tilde{h}\tilde{\nu}$ with $\tilde{h}>0$, and thus $\widetilde{\Sigma}$ is degenerate stable. 
	Since $\tilde{h}>0$, $\tilde{h}\tilde{\nu}$ can not be reduced to $\Sigma$, so $\Sigma$ is strictly stable. 
	Finally, if $m\geq 2$ and $\mathcal{Y}=\emptyset$, then $\widetilde{\Sigma}_i$ can be written as $m$ graphs over $\widetilde{\Sigma}$. 
	Note each graph is $G$-connected and $\widetilde{\Sigma}_i$ has at most two $G$-connected components. 
	Hence, we must have $m=2$ and $\widetilde{\Sigma}_i$ has exactly two $G$-connected components each of which is diffeomorphic to $\widetilde{\Sigma}$ as well as $\Sigma_i$. 
\end{proof}

Note Theorem \ref{Thm: Jacobi field} is also valid for varying metrics as we mentioned in Remark \ref{Rem: various metrics}. 
Additionally, Theorem \ref{Thm: 3}(i)-(iii) follow immediately from Theorem \ref{Thm: compactness theorem} and \ref{Thm: Jacobi field}.
Moreover, combining Theorem \ref{Thm: bumpy metric}, \ref{Thm: compactness theorem}, and \ref{Thm: Jacobi field}, we have Corollary \ref{Cor: generic finiteness} the generic finiteness result for area uniformly bounded min-max $(c,G)$-hypersurfaces. 

\begin{proof}[Proof of Corollary \ref{Cor: generic finiteness}]
	Let $g_{_M}$ be a $G$-bumpy metric on $M$, which is $C^\infty_G$-generic by Theorem \ref{Thm: bumpy metric}. 
	If there are infinitely many min-max $(c,G)$-hypersurfaces $\{\Sigma_i\}_{i\in\N}$ with $\|\Sigma_i\|(M)\leq C$. 
	Then Theorem \ref{Thm: compactness theorem} and \ref{Thm: Jacobi field} give a limit min-max $(c,G)$-hypersurface $\Sigma$ with a non-trivial Jacobi field on $\Sigma$ or its double cover, which contradicts the choice of $g_{_M}$. 
\end{proof}


\section{$G$-index estimates for min-max $G$-hypersurfaces}\label{Sec: G-index estimates}
In this section, we show the $G$-index upper bounds for minimal $G$-hypersurfaces given in Theorem \ref{Thm: min-max} using an equivariant version of the Deformation Theorem \cite[Theorem 5.8]{marques2016morse}. 
First, we need the following definition similar to \cite[Definition 4.1]{marques2016morse}, which generalizes the equivariant Morse index to the varifold sense. 

\begin{definition}\label{Def: unstable varifolds}
	Given a stationary $G$-varifold $\Sigma\in \V_n^G(M)$ and a nonnegative number $\epsilon\geq 0$, we say $\Sigma$ is {\em $(G,k)$-unstable in an $\epsilon$-neighborhood} if there exist $0<c_0<1$ and a smooth family $\{F_v\}_{v\in \overline{\mathbb{B}}^k_1}\subset {\rm Diff}_G(M)$ of $G$-equivariant diffeomorphisms such that:
	\begin{itemize}
		\item[(i)] $F_0 = id$, $F_{-v} = F_v^{-1}$, for all $v\in \overline{\mathbb{B}}^k_1$;
		\item[(ii)] for any $V\in \overline{\mathbf{B}}_{2\epsilon}^\mF (\Sigma)$, the smooth function
		\[A^V:  \overline{\mathbb{B}}^k_1\rightarrow [0,\infty), \qquad A^V(v) = ||(F_v)_\# V||(M), \]
		has a unique maximum at $m(V)\in \mathbb{B}^k_{c_0/\sqrt{10}}$, and $-\frac{1}{c_0} {\rm Id} \leq D^2 A^V(u) \leq -c_0 {\rm Id} $ for all $u\in \overline{\mathbb{B}}^k_1$. 
	\end{itemize}
	Moreover, we say a $G$-varifold $\Sigma\in \V_n^G(M)$ is {\em $(G,k)$-unstable} if it is stationary and $(G,k)$-unstable in an $\epsilon$-neighborhood for some $\epsilon>0$. 
\end{definition}

\begin{remark}
	In the above definition, we have $m(\Sigma) = 0$ since $\Sigma$ is stationary. 
	Additionally, the function $A^{V_i}$ tends to $A^V$ in the smooth topology provided $V_i\to V$ in the $\mF$-topology (\cite[Section 2.3]{pitts2014existence}). 
	Hence, if $\Sigma$ is $(G,k)$-unstable in a $0$-neighborhood, then $\Sigma$ is $(G,k)$-unstable in an $\epsilon $-neighborhood for some $\epsilon >0$. 
\end{remark}

\begin{lemma}\label{Lem: unstable varifold}
	If $\Sigma\in\IV_n^G( M)$ is induced by a closed smooth embedded minimal $G$-hypersurface, then $\Sigma$ is $(G,k)$-unstable if and only if ${\rm Index}_G(\spt(\|\Sigma\|))\geq k$. 
\end{lemma}
\begin{proof}
	If ${\rm Index}_G(\spt(\|\Sigma\|))\geq k$. 
	Let $\{X_i\}_{i=1}^k\subset \mathfrak{X}^{\bot, G}(\Sigma)$ be $L_2$-orthogonal vector fields with $\delta^2\Sigma(X_i)<0$, which extend to $G$-vector fields $\{Y_i\}_{i=1}^k$ by (\ref{Eq: average vector field}). 
	Then for $\delta>0$ sufficiently small, one easily verifies that $F_v:= \phi_1^{\delta\sum_{i=1}^kv_iY_i}\in {\rm Diff}_G(M)$ and $\Sigma$ is $(G,k)$-unstable, where $v\in \overline{\mathbb{B}}^k_1$ and $\{\phi^Y_t\}$ are the diffeomorphisms generated by $Y$. 
	
	Meanwhile, if $\Sigma$ is $(G,k)$-unstable, let $\{F_v\}_{v\in \overline{\mathbb{B}}^k}\subset {\rm Diff}_G(M)$ be given by Definition \ref{Def: unstable varifolds}. 
	Then we must have $\{X_i(x) = (\frac{\partial F}{\partial v_i}(0,x))^\perp\}_{i=1}^k \subset \mathfrak{X}^{\bot, G}(\Sigma)$ are linear independent and $\delta^2\Sigma(X_i)<0$, which indicates ${\rm Index}_G(\spt(\|\Sigma\|))\geq k$. 
\end{proof}

\begin{proposition}\label{Prop: index estimate}
	Let $\{ \Sigma_i\}_{i\in\N}$ and $\Sigma_\infty$ be the min-max $(c,G)$-hypersurfaces in Theorem \ref{Thm: compactness theorem}. 
	Suppose $\sup_{i\in\N} {\rm Index}_G(\spt(\|\Sigma_i \|) )\leq k<\infty$. 
	Then ${\rm Index}_G(\spt(\|\Sigma_\infty \|))\leq k$. 
\end{proposition}

\begin{proof}
	Suppose ${\rm Index}_G(\spt(\|\Sigma_\infty \|))\geq k +1$. 
	Then by Lemma \ref{Lem: unstable varifold}, $\Sigma_\infty$ is $(G, k+1)$-unstable in an $\epsilon$-neighborhood for some $\epsilon>0$. 
	By the varifold convergence, we have $\Sigma_i \in \overline{\mathbf{B}}_{2\epsilon}^\mF (\Sigma_\infty)$ for $i$ sufficiently large. 
	Hence, by definitions, $\Sigma_i$ is $(G, k+1)$-unstable for $i$ large enough, which contradicts the assumption and Lemma \ref{Lem: unstable varifold}. 
\end{proof}
\begin{remark}\label{Rem: index estimate}
	As in Remark \ref{Rem: various metrics}, suppose additionally that the $G$-index of $\spt(\|\Sigma_i\|)$ is given under the Riemannian metric $g_{_{M}}^i$, where $\{g_{_M}^i\}_{i\in \N}$ is a sequence of metrics with $\lim_{k\to\infty}g_{_M}^i=g_{_M}$ in the smooth topology. 
	Then we also have the estimate ${\rm Index}_G(\spt(\|\Sigma_\infty\|))\leq k$ under the metric $g_{_M}$. 
	This is because if $D^2A^V$ is strictly negative definite with respect to the metric $g_{_M}$, then it is also strictly negative definite with respect to the metric $g_{_M}^i$ for $i$ large enough. 
\end{remark}

\begin{lemma}\label{Lem: flow in parameter space}
	Let $\Sigma\in\V_n^G(M)$ be a stationary $G$-varifold that is $(G,k)$-unstable in an $\epsilon$-neighborhood. 
	Let $\{F_v\}_{v\in \overline{\mathbb{B}}^k}$ and $A^V$ be defined as in Definition \ref{Def: unstable varifolds}.  
	For any $V\in  \overline{\mathbf{B}}_{2\epsilon}^\mF (\Sigma)$, denote $\{\phi^V(\cdot, t)\}_{t\geq 0}\subset {\rm Diff}(\overline{\mathbb{B}}^k)$ as the flow generated by 
	\[u \mapsto -(1-|u|^2)\nabla A^V(u),\qquad u\in \overline{\mathbb{B}}^k.\]
	Then for any $\delta<\frac{1}{4}$, there exists a time $T=T(\delta,\epsilon,\Sigma, \{F_v\}, c_0 )\geq 0$ such that: if $V\in \overline{\mathbf{B}}_{2\epsilon}^\mF (\Sigma)$, $v\in \overline{\mathbb{B}}^k$, and $|v - m(V)|\geq \delta$, then 
	\[A^{V}\left(\phi^{V}(v, T)\right)<A^{V}(0)-\frac{c_{0}}{10}, \quad \text { and } \quad\left|\phi^{V}(v, T)\right|>\frac{c_{0}}{4}.\] 
\end{lemma}
\begin{proof}
	The proof is the same as \cite[Lemma 4.5]{marques2016morse}. 
\end{proof}

Now, we show the following deformation theorem in the equivariant setting, which generalizes the Deformation Theorem A in \cite{marques2016morse}. 
Roughly speaking, we can use the following deformation theorem to generate a new sequence of maps that is homotopic to $\{\Phi_i\}_{i\in\N}$ and far away from a minimal $G$-hypersurface with a large $G$-index. 
\begin{theorem}[Deformation Theorem]\label{Thm: deformation theorem}
	Let $\{\Phi_i\}_{i\in\N}$ be a sequence of continuous maps from $X$ to $\Z_n^G(M;\mF;\mZ_2 )$, where $X$ is a cubical complex of dimension $k$. 
	Denote 
	\[
		L=\mathbf{L}(\{\Phi_i\}_{i\in\N} ) := \limsup_{i\to\infty} \sup_{x\in X}\M(\Phi_i(x)). 
	\]
	Suppose 
	\begin{itemize}
		\item[(1)] $\Sigma\in\V_n^G(M)$ is stationary and $(G, k+1)$-unstable;
		\item[(2)] $\mathcal{K}\subset \V_n(M)$ is a subset with $\mF( \Sigma, \mathcal{K})>0$ and $\mF(|\Phi_i|(X), \mathcal{K}) >0$ for all $i\geq i_0$;
		\item[(3)] $||\Sigma||(M) = L$.
	\end{itemize}
	Then there exist $\bar{\epsilon}>0$, $j_0\in\N$, and a sequence $\{\Psi_i\}_{i\in\N}$ of maps from $X$ to $\Z_n^G(M;\mF;\mZ_2 )$, such that
	\begin{itemize}
		\item[(i)] $\Psi_i$ is homotopic to $\Phi_i$ in $\Z_n^G(M;\mF;\mZ_2)$ for all $i\in\N$;
		\item[(ii)] $\mathbf{L}(\{\Psi_i\}_{i\in\N})\leq L$;
		\item[(iii)] $\mF(|\Psi_i|(X),  ~\overline{\mathbf{B}}^\mF_{\bar{\epsilon}}(\Sigma)\cup \mathcal{K} ) > 0$ for all $i\geq j_0$. 
	\end{itemize}
\end{theorem}	

\begin{proof}
	We follow the arguments in \cite[Theorem 5.1]{marques2016morse} and sketch the main steps for the sake of completeness. 
	First, set $d:=  \mF( \Sigma, \mathcal{K})>0$. 
	Since $\mF( \Sigma, \mathcal{K})>0$ and $\Sigma$ is $(G,k+1)$-unstable, let $\epsilon>0$ and $\{F_v\}_{v\in \overline{\mathbb{B}}^{k+1}}\subset {\rm Diff}_G(M)$ be given in Definition \ref{Def: unstable varifolds} corresponding to $\Sigma$ so that 
	\begin{equation}\label{Eq: away from K}
		\mF(\left(F_{v}\right)_{\#} V, \mathcal{K})>\frac{d}{2}, \qquad\forall V\in \overline{\mathbf{B}}_{2\epsilon}^\mF (\Sigma), ~v\in \overline{\mathbb{B}}^{k+1}.
	\end{equation} 
	Then for every $i\in\N$, we can take a sufficiently fine subdivision $X(k_i)$ of $X$ so that 
	\[ \mF( |\Phi_i(x)|, |\Phi_i(y)| ) < \delta_i, \quad\mbox{for all $x,y$ in a same cell of $X(k_i)$},\]
	where $\delta_i := \min \{2^{-(i+k+2)}, \epsilon/4 \}$. 
	
	For any $\eta>0$, denote by $U_{i,\eta}$ the subcomplex of $X(k_i)$ formed by the cells $\sigma\in X(k_i)$ with $\mF(|\Phi_i(x)|, \Sigma)<\eta$ for all $x\in \sigma$. 
	Therefore, if $\beta\notin U_{i,\eta}$ and $y\in \beta$, then 
	\begin{equation}\label{Eq: far away from Sigma}
		\mF(|\Phi_i(y)|, \Sigma)\geq \eta - \delta_i.
	\end{equation}
	For any $x\in X(k_i)$, denote $A_i^x := A^{|\Phi_i(x)|} $, $m_i(x) := m(|\Phi_i(x)|)$, and $ \phi_i^x := \phi^{|\Phi_i(x)|	}$, where $A,m,\phi$ are defined in Definition \ref{Def: unstable varifolds} and Lemma \ref{Lem: flow in parameter space}. 
	It now follows from the construction in \cite[Theorem 5.1]{marques2016morse} that there is a homotopy map $\hat{H}_i: U_{i,2\epsilon}\times [0,1]\to \mathbb{B}^{k+1}_{2^{-i}}(0)$ so that $\hat{H}_i(x,0) = 0$ for all $x\in U_{i,2\epsilon}$, and 
	\begin{equation}\label{Eq: homotopy away from m}
		\inf _{x \in U_{i, 2 \epsilon}} |m_{i}(x)-\hat{H}_{i}(x, 1) | \geq \eta_{i}>0 \quad \mbox{for some } \eta_{i}>0. 
	\end{equation}
	Let $c: [0,\infty)\to [0,1]$ be a non-increasing cut-off function so that $c=1$ in $ [0, \frac{3\epsilon}{2} ]$ and $c=0$ in $[\frac{7\epsilon}{4},\infty)$. 
	Because $\delta_i\leq \frac{\epsilon}{4}$ and (\ref{Eq: far away from Sigma}), if $y\notin U_{i,2\epsilon}$, then $\mF(|\Phi_i(y)|, \Sigma)\geq\frac{7\epsilon}{4}$. 
	Hence, $c(\mF(|\Phi_i(y)|, \Sigma)) =0 $ provided $y\notin U_{i,2\epsilon}$. 
	
	Now we define a continuous homotopy map $H_i: X\times [0,1] \to \mathbb{B}^{k+1}_{2^{-i}}(0)$ by 
	\begin{equation*}
		H_i(x,t) := \left\{\begin{array}{ll}{ \hat{H}_i(x, ~t\cdot c(\mF(|\Phi_i(x)|, \Sigma))) ,} & {x\in U_{i,2\epsilon}} \\ {0, }&{x \notin U_{i,2\epsilon}}\end{array}\right. .
	\end{equation*}
	For $\eta_i$ in (\ref{Eq: homotopy away from m}), let $T_i=T_i(\eta_i ,\epsilon,\Sigma, \{F_v\}, c_0 )$ be given by Lemma \ref{Lem: flow in parameter space}.  
	Then we have a continuous map $D_i: X\to \overline{\mathbb{B}}^{k+1}$, $D_i(x):= \phi_i^x \big(H_i(x), T_i\cdot c(\mF(|\Phi_i(x)|, \Sigma)) \big)$, so that $D_i(x) =0 $ if $x\notin U_{i,2\epsilon}$. 
	Define then 
	\[ \Psi_i:X\to \Z_n^G(M; \mF; \mZ_2), \quad \Psi_i(x) := (F_{D_i(x)})_\#(\Phi_i(x)) .\]
	For each $i\in\N$, it follows from the constructions that $\Psi_i$ is homotopic to $\Phi_i$ in $\Z_n^G(M; \mF; \mZ_2)$ and $\Psi_i(x) = \Phi_i(x) $ if $ x\notin U_{i, 2\epsilon}$. 
	One can also follow the proof in \cite[Theorem 5.1]{marques2016morse} to verify (ii) and (iii). 
\end{proof}

Combining the Deformation Theorem \ref{Thm: deformation theorem} and the Compactness Theorem \ref{Thm: compactness theorem}, we can now show our main theorem of $G$-index estimates. 

\begin{theorem}\label{Thm: main theorem}
	Let $(M^{n+1}, g_{_M})$ be a closed Riemannian manifold with a compact Lie group $G$ acting as isometries so that $3\leq {\rm codim}(G\cdot p)\leq 7$ for all $ p\in M$. 
	For any $k$-dimensional cubical complex $X$, let ${\bm \Pi}\in \big[ X^k, \Z_{n}^G(M;\mF;\mZ_2)\big]$ be a $G$-homotopy class. 
	Then there is a min-max $((3^k)^{3^k},G)$-hypersurface $\Sigma\in\IV^G_n(M)$ so that:
	\begin{itemize}
		\item[(i)] $\spt(\|\Sigma\|)$ is a closed smooth embedded minimal $G$-hypersurface in $M$;
		\item[(ii)] $\mathbf{L}({\bm \Pi}) = ||\Sigma||(M) $;
		\item[(iii)] ${\rm Index}_G(\spt(\|\Sigma\|))\leq k$.
	\end{itemize}
	Moreover, if the metric $g_{_M}$ is $G$-bumpy, then for any min-max sequence $\{\Phi_i\}_{i\in\N}\subset{\bm \Pi}$, we can further take $\Sigma \in  \mathbf{C}(\{\Phi_i\}_{i\in\N} )$. 
\end{theorem}

\begin{proof}
	The proof is divided into two cases depending on whether the metric $g_{_M}$ is $G$-bumpy. 
	For simplicity, let $c:=(3^k)^{3^k}$. 
	
	{\bf Case 1: $g_{_M}$ is $G$-bumpy.} 
	Define $\mathcal{W}_G$ to be the set of all integral stationary $G$-varifolds $V\in\IV_n^G(M)$ such that  
	\[
		\mbox{$||V||(M) = \mathbf{L}({\bm \Pi}) $ and $V$ is a min-max $(c,G)$-hypersurface (with multiplicity)}. 
	\]
	Let $\mathcal{W}_G^{k+1}:=\{V\in \mathcal{W}_G: {\rm Index}_G(\spt(\|V\|)) \geq k+1 \}$ be a subset of $\mathcal{W}_G$. 
	By the $G$-bumpy property of $g_{_M}$ and Corollary \ref{Cor: generic finiteness}, $\mathcal{W}_G$ and $\mathcal{W}_G^{k+1}$ are both finite sets.

	For any min-max sequence $\{\Phi_i\}_{i\in\N}\subset {\bm \Pi}$ and $r>0$, denote 
	\[
		\mathcal{W}(r):= \{V\in \mathcal{W}_G : \mF(V, \mathbf{C}(\{\Phi_i\}_{i\in\N}) ) \geq r \}.
	\]
	Then we claim that $\mF(|\Phi_i|(X), \mathcal{W}(r)) > \epsilon_0$ for some $\epsilon_0>0$ and all $i$ sufficiently large. 
	Otherwise, suppose there exist $x_{i_j}\in X$ and $V_j\in \mathcal{W}(r) $ with $ \mF(|\Phi_{i_j}(x_{i_j})|, V_j) \leq \frac{1}{j}$ for all $j=1,2,\dots$. 
	Then we have $\lim_{j\to\infty}||\Phi_{i_j}(x_{i_j})||(M) = \mathbf{L}({\bm \Pi})$, and thus $\lim_{j\to\infty} V_j = \lim_{j\to\infty}|\Phi_{i_j}(x_{i_j})| \in  \mathbf{C}(\{\Phi_i\}_{i\in\N})$ up to a subsequence, which is a contradiction. 
	
	Moreover, since $\mathcal{W}_G^{k+1}$ is finite (Corollary \ref{Cor: generic finiteness}), we can write $\mathcal{W}_G^{k+1} \setminus \overline{\mathbf{B}}_{\epsilon_0}^\mF ( \mathcal{W}(r) ) = \{\Sigma_1,\Sigma_2,\dots \Sigma_J\}$. 
	By Lemma \ref{Lem: unstable varifold}, each $\Sigma_j$ is $(G,k+1)$-unstable. 
	Therefore, we can apply the deformation theorem (Theorem \ref{Thm: deformation theorem}) repeatedly to eliminate every $\Sigma_j$. 
	
	Specifically, by applying Theorem \ref{Thm: deformation theorem} with $\mathcal{K} = \overline{\mathbf{B}}_{\epsilon_0}^\mF ( \mathcal{W}(r) )$ and $\Sigma = \Sigma_1 $, we obtain $\epsilon_1>0$, $i_1\in\N$, and a sequence $\{\Phi_i^1\}_{i\in\N}$ of maps from $X$ to $\Z_n^G(M;\mF;\mZ_2 )$, such that
	\begin{itemize}
		\item $\Phi_i^1$ is homotopic to $\Phi_i$ in $\Z_n^G(M;\mF;\mZ_2)$ for all $i\in\N$;
		\item $\mathbf{L}(\{\Phi_i^1\}_{i\in\N})\leq L$;
		\item $\mF\big(|\Phi_i^1|(X),  ~\overline{\mathbf{B}}^\mF_{\epsilon_1}(\Sigma_1) \cup  \overline{\mathbf{B}}_{\epsilon_0}^\mF ( \mathcal{W}(r) ) \big) > 0$ for all $i\geq i_1$; 
		\item no $\Sigma_j$ belongs to $\partial \overline{\mathbf{B}}^\mF_{\epsilon_1}(\Sigma_1) $,
	\end{itemize}
	where $L :=\mathbf{L}(\{\Phi_i\}_{i\in\N} ) = {\bf L}({\bm \Pi})$. 
	
	Now, if $\Sigma_2\notin \overline{\mathbf{B}}^\mF_{\epsilon_1}(\Sigma_1)$, we can apply Theorem \ref{Thm: deformation theorem} with $\mathcal{K} = \overline{\mathbf{B}}^\mF_{\epsilon_1}(\Sigma_1) \cup  \overline{\mathbf{B}}_{\epsilon_0}^\mF ( \mathcal{W}(r) ) $ and $\Sigma = \Sigma_2 $ to get $\epsilon_2>0$, $i_2\in\N$, and a sequence $\{\Phi_i^2\}_{i\in\N}$, such that
	\begin{itemize}
		\item $\Phi_i^2$ is homotopic to $\Phi_i$ in $\Z_n^G(M;\mF;\mZ_2)$ for all $i\in\N$;
		\item $\mathbf{L}(\{\Phi_i^2\}_{i\in\N})\leq L$;
		\item $\mF \big( |\Phi_i^2|(X),  ~\overline{\mathbf{B}}^\mF_{\epsilon_2}(\Sigma_2) \cup \overline{\mathbf{B}}^\mF_{\epsilon_1}(\Sigma_1) \cup  \overline{\mathbf{B}}_{\epsilon_0}^\mF ( \mathcal{W}(r) ) \big) > 0$ for all $i\geq i_2$; 
		\item no $\Sigma_j$ belongs to $\partial \overline{\mathbf{B}}^\mF_{\epsilon_2}(\Sigma_2)  \cup \partial \overline{\mathbf{B}}^\mF_{\epsilon_1}(\Sigma_1) $.
	\end{itemize}
	If $\Sigma_2\in \mathbf{B}^\mF_{\epsilon_1}(\Sigma_1)$, we just skip to the procedure for $\Sigma_3$.

	By repeating this procedure, we stop after a finite number of times, say $m$ times. 
	Then we get positive numbers $\{\epsilon_q \}_{q=1}^m$, $i_m\in\N$, $\{j_q\}_{q=1}^m\subset\N$, and a sequence $\{\Phi_i^m\}_{i\in\N}$ such that 
	\begin{itemize}
		\item $\Phi_i^m$ is homotopic to $\Phi_i$ in $\Z_n^G(M;\mF;\mZ_2)$ for all $i\in\N$;
		\item $\mathbf{L}(\{\Phi_i^m\}_{i\in\N})\leq L$;
		\item $\mF \big( |\Phi_i^m|(X),  ~ \cup_{q=1}^m \overline{\mathbf{B}}^\mF_{\epsilon_q}(\Sigma_{j_q}) \cup  \overline{\mathbf{B}}_{\epsilon_0}^\mF ( \mathcal{W}(r) ) \big) > 0$ for all $i\geq i_m$; 
		\item $\mathcal{W}_G^{k+1} \setminus \overline{\mathbf{B}}_{\epsilon_0}^\mF ( \mathcal{W}(r) ) = \{\Sigma_1,\Sigma_2,\dots \Sigma_J\} \subset \cup_{q=1}^m \mathbf{B}^\mF_{\epsilon_q}(\Sigma_{j_q})$.
	\end{itemize}
	Define the sequence $\{\Psi_l\}_{l\in\N}$ by $\Psi_l := \Phi_{l}^m $, which clearly satisfies: 
	\begin{itemize}
		\item[(1)] $\Psi_l$ is homotopic to $\Phi_{l}$ in $\Z_n^G(M;\mF;\mZ_2)$ for all $l\in\N$;
		\item[(2)] $\mathbf{L}(\{\Psi_l\}_{l\in\N}) = L = {\bf L}({\bm \Pi})$;
		\item[(3)] $\mathbf{C}(\{\Psi_l\}_{l\in\N}) \cap (\mathcal{W}_G^{k+1}  \cup \mathcal{W}(r)) = \emptyset$.
	\end{itemize}
	Now, we can apply the Equivariant Min-max Theorem \ref{Thm: min-max} to get an integral stationary $G$-varifold $\Sigma_r\in \mathbf{C}(\{\Psi_l\}_{l\in\N})\cap \big(\mathcal{W}_G \setminus (\mathcal{W}_G^{k+1}  \cup \mathcal{W}(r)) \big)$ as required. 
	
	Moreover, by Compactness Theorem \ref{Thm: compactness theorem} and Proposition \ref{Prop: index estimate}, there is a sequence $r_i\to 0$ so that $\Sigma_{r_i}$ converges in the varifold sense to an integral stationary $G$-varifold $\Sigma\in \mathbf{C}(\{\Phi_i\}_{i\in\N})$, which is a min-max $(c,G)$-hypersurface with ${\rm Index}_G(\spt(\|\Sigma \|))\leq k$. 
	Indeed, since $\mathcal{W}_G$ is finite, $\Sigma_{r_i} = \Sigma$ for $i$ large enough.

	{\bf Case 2: $g_{_M}$ is not $G$-bumpy.} 
	By Theorem \ref{Thm: bumpy metric}, there exists a sequence of $G$-bumpy metrics $\{g_{_M}^j\}_{j\in \N}$ tending to $g_{_M}$ in the smooth topology. 
	As we proved in {\bf Case 1}, for each $j\in\N$, we have a min-max sequence $\{\Psi_l^j\}_{l\in\N}$ of ${\bf \Pi}$ with respect to the metric $g_{_M}^j$, and an integral $g_{_M}^j$-stationary $G$-varifold $\Sigma_j\in \mathbf{C}(\{\Psi_l\}_{l\in\N})$ such that 
	\begin{itemize}
		\item $\Sigma_j$ is a min-max $(c,G)$-hypersurface corresponding to ${\bf \Pi}$ under the metric $g_{_M}^j$,
		\item $||\Sigma_j||(M)= L_j$, ${\rm Index}_G(\spt(\|\Sigma_j\|))\leq k$,
	\end{itemize}
	where $L_j$ is the width of ${\bf \Pi}$ with respect to $g_{_M}^j$. 
	It is clear that $\lim_{j\to\infty}L_j = \mathbf{L}({\bm \Pi}) $. 
	Therefore, by Compactness Theorem \ref{Thm: compactness theorem} and Remark \ref{Rem: various metrics}, \ref{Rem: index estimate}, there is a (not relabeled) subsequence $j\to \infty$ so that $\Sigma_j$ converges to an integral stationary $G$-varifold $\Sigma$, which is a min-max $(c,G)$-hypersurface under the metric $g_{_M}$ satisfying (i)-(iii). 
\end{proof}

By Proposition \ref{Prop: isomorphism}, we have $H^1(\Z_n^G(M; \F; \mZ_2); \mZ_2 ) \cong \mZ_2$ with a generator $\bar{\lambda}$. 
For any $p\in\mZ^+$, denote by $\mathcal{P}^G_p$ the set of all continuous maps $\Phi: X^k\to \Z_n^G(M; \mF ; \mZ_2)$ such that $0\neq \Phi^*(\bar{\lambda}^p) \in H^p(X; \mZ_2)$, where $X^k$ is any cubical complex of dimension $k \geq p$ and $\bar{\lambda}^p$ is the cup product of $\bar{\lambda}$ with itself $p$ times. 
Then the {\em $(G,p)$-width} of $M$ is defined by 
\begin{equation}\label{Eq: Gp-width}
	\omega_p^G(M) := \inf_{\Phi\in \mathcal{P}_p^G} \sup_{x\in{\rm dmn}(\Phi) } \M(\Phi(x)). 
\end{equation}

\begin{proof}[Proof of Corollary \ref{Cor: achieve width of M}]
	For any $\Phi \in \mathcal{P}^G_p$ with $X = {\rm dmn}(\Phi)$, let $X^{(p)}$ be the $p$-skeleton of $X$, $i:X^{(p)}\to X$ be the natural inclusion map. 
	Then we have $H^p(X, X^{(p)};\mZ_2) = 0$. 
	Thus the exact cohomology sequence
	$ H^{p}\left(X, X^{(p)} ; \mathbb{Z}_{2}\right) \stackrel{j^{*}}{\rightarrow} H^{p}\left(X ; \mathbb{Z}_{2}\right) \stackrel{i^{*}}{\rightarrow} H^{p}\left(X^{(p)} ; \mathbb{Z}_{2}\right) $ implies 
	that $\Phi\circ i : X^{(p)} \to \Z_n^G(M; \mF ; \mZ_2) $ is also an element in $\mathcal{P}_p^G$. 

	Let $\Phi_i: X_i \to \Z_n^G(M; \mF; \mZ_2)$ be a sequence in $\mathcal{P}_p^G$ so that $ \sup_{x\in X_i } \M(\Phi_i(x)) \to \omega_p^G(M)$. 
	Then, $\Phi_i\llcorner X_i^{(p)} \in \mathcal{P}_p^G$ and $\sup_{x\in X_i^{(p)} } \M(\Phi_i(x)) \to \omega_p^G(M) $. 
	Denote by ${\bm \Pi}_i\in [X_i^{(p)} , \Z_n^G(M; \mF; \mZ_2)]$ the $G$-homotopy class of $\Phi_i\llcorner X_i^{(p)}$. 
	Then ${\bf L}({\bm \Pi}_i) \to \omega^G_p(M)$. 
	Finally, Corollary \ref{Cor: achieve width of M} follows directly from Theorem \ref{Thm: main theorem}, \ref{Thm: compactness theorem}, and Proposition \ref{Prop: index estimate}. 
\end{proof}

\appendix
	\addcontentsline{toc}{section}{Appendices}
	\renewcommand{\thesection}{\Alph{section}}

\section{Local $G$-invariant minimal foliations}\label{Sec: foliation}

In this appendix, we construct local $G$-invariant minimal foliations around an orbit $G\cdot p$ in a minimal $G$-hypersurface $\Sigma$. 
This is an equivariant generalization of the proposition in \cite[Appendix]{white1987curvature}. 
We start with some notations and preparatory work. 

Let $\Sigma$ be a closed smooth embedded $G$-invariant minimal hypersurface with a $G$-invariant unit normal $\nu$. 
For any $G\cdot p\subset \Sigma$, define 
\[\B_{r_0}^\Sigma(G\cdot p) := \{q\in\Sigma: \dist_\Sigma(q,G\cdot p)< r_0 \}\]
to be a geodesic $r_0$-tube around $G\cdot p$ in $\Sigma$. Then for $r_0>0$ small enough, 
\[ E: \B_{r_0}^\Sigma(G\cdot p)\times [-r_0, r_0] \to M, \quad E(q,t):= \exp_q(t\cdot \nu(q)), \]
is an equivariant diffeomorphism from $\B_{r_0}^\Sigma(G\cdot p)\times [-r_0, r_0]$ to a neighborhood of $G\cdot p $ in $M$: 
\[\C_{r_0}(G\cdot p) :=E(\B_{r_0}^\Sigma(G\cdot p)\times [-r_0, r_0]).\]
Next, let $\mS_{p,r_0}^\Sigma := \{ q\in \B_{r_0}^\Sigma(G\cdot p): \dist_\Sigma(q,  p)=\dist_\Sigma(q,G\cdot p) \}$ be a slice of $G\cdot p$ at $p$ in $\Sigma$. 
Then we denote 
\[ \mS_{p,r_0} := E(\mS_{p,r_0}^\Sigma\times [-r_0, r_0]) \]
to be a slice of $G\cdot p$ at $p$ in $\C_{r_0}(G\cdot p)$. 
By \cite[Section 2.1.3 ($\Sigma_5$)]{berndt2016submanifolds}, $\C_{r_0}(G\cdot p)$ is a fiber bundle over $G\cdot p$ with fiber $\mS_{p,r_0}$ at $p$. 
Let 
\[ P : \C_{r_0}(G\cdot p) \to G\cdot p, \quad P(q)=p,~\forall q\in \mS_{p,r_0}, \]
be the bundle projection, which is clearly a submersion. 

Next, we are going to construct a $G$-invariant elliptic integrand $\varphi$ on $\mS_{p,r_0}$ so that for any $G$-invariant hypersurface $\Gamma$ in $\C_{r_0}(G\cdot p)$, $\mH^n(\Gamma) = \mH^{\dim(G\cdot p)}(G\cdot p)\cdot {\underline{\underline \varphi}}(\Gamma\cap \mS_{p,r_0} )$. 
Once we find such $\varphi$, then the $G$-invariant minimal foliation on $\C_{r}(G\cdot p)$ is equivalent to the $G_p$-invariant $\varphi$-stationary foliation on $\mS_{p,r_0}\cong \mathbb{B}_{r_0}^{n-\dim(G\cdot p)}(0)\times [r_0,r_0]$. 
To construct such an elliptic integrand, we need the following lemma.

\begin{lemma}\label{Lem: distribution on slice}
	There exists a smooth distribution $V$ on $\C_{r_0}(G\cdot p) $ with rank $\dim(G\cdot p)$ so that 
	\begin{itemize}
		\item[(i)] $V$ is $G$-invariant, i.e. $dg(V(q)) = V(g\cdot q)$;
		\item[(ii)] $T_qG\cdot q = T_qG_p\cdot q \oplus V(q)$ for all $q\in \mS_{p,r_0}$.
	\end{itemize}
\end{lemma}
\begin{proof}
	Since $G$ is a compact Lie group, we can take a bi-invariant metric $g_{_G}$ on $G$ (\cite[Corollary 1.4]{milnor1976curvatures}). 
	Then for any $g\in G$, the conjugate map 
	\[c(g) :G\to G,\quad c(g)(h) := g\cdot h\cdot g^{-1},\]
	is an isometry on $(G,g_{_G})$. 
	Let $\rho: G\times M\to M$, $\rho(g,q):= g\cdot q$, be the smooth $G$-action map and $\rho_q = \rho(\cdot, q): G\to M $ for any $q\in M$. 
	Then we have 
	\[ (d\rho_q)_e = (d\rho)_{(e,q)}(\cdot, 0): T_eG\to T_qM, \quad {\rm and}\quad \ker((d\rho_q)_e) = T_eG_q, \]
	where $e\in G$ is the identity element. 
	
	For any $q\in \C_{r_0}(G\cdot p)$, assume without loss of generality that $q\in \mS_{p,r_0}$. 
	Therefore, $G_q\subset G_p$ and $\ker((d\rho_q)_e)\subset \ker((d\rho_p)_e) $. 
	Let $V_e := (T_eG_p)^\perp = (\ker((d\rho_p)_e))^\perp$ be the orthogonal complement of $T_eG_p$ in $T_eG$ with respect to the metric $g_{_G}$. 
	Define then 
	\[ V(q):= (d\rho_q)_e(V_e) = (d\rho)_{(e,q)}(V_e,0) ,\]
	which is a smooth distribution of rank $\dim(G\cdot p)=\dim(V_e)$ since $(d\rho_q)_e \llcorner V_e$ is injective. 
	
	For any $g\in G$, since $G_{g\cdot p} = g G_p g^{-1} = c(g)(G_p)$ and $c(g)$ is an isometry, we see
	\[ (dc(g))_e\left((T_eG_p)^\perp\right) = \big((dc(g))_e(T_eG_p)\big)^\perp = (T_eG_{g\cdot p})^\perp.\]
	Additionally, if $v=\frac{d}{dt}\big|_{t=0} g(t) \in (T_eG_p)^\perp$, 
	then 
	\begin{eqnarray*}
		(dg\circ d\rho_q)_e(v) &=& \frac{d}{dt}\Big|_{t=0} g\cdot g(t)\cdot q = \frac{d}{dt}\Big|_{t=0} g\cdot g(t)\cdot g^{-1} \cdot gq
		\\ &=& \frac{d}{dt}\Big|_{t=0} c(g)( g(t))\cdot gq = (d\rho_{gq})_e\circ (dc(g))_e(v). 
	\end{eqnarray*}
	Thus, $ dg(V(q)) = (dg\circ d\rho_q)_e((T_eG_p)^\perp) =(d\rho_{gq})_e\circ (dc(g))_e((T_eG_p)^\perp) = (d\rho_{gq})_e((T_eG_{g\cdot p})^\perp) = V(gq) $, and $V$ is $G$-invariant. 
	
	Since $T_qG_p\cdot q = (d\rho_q)_e(T_eG_p) $ and $T_qG\cdot q = (d\rho_q)_e(T_eG)$, we have $T_qG_p \cdot q + V(q) = T_qG\cdot q$. 
	Since $\ker((d\rho_q)_e)\subset T_eG_p$, we see $T_qG_p \cdot q \cap V(q) = \{0\}$, which implies (ii). 
\end{proof}

Let $V$ be given by Lemma \ref{Lem: distribution on slice}. 
Then for any $q\in \mS_{p,r_0}$ and $0\neq v\in T_q\mS_{p,r_0}$, define 
\begin{equation}\label{Eq: integrand on slice}
	\varphi: T\mS_{p,r_0}\setminus \mS_{p,r_0}\times\{0\} \to \mathbb{R},\quad \varphi(q,v) :=  \left\{\begin{array}{ll}{|v| ,} & {\dim(G\cdot p)=0} \\ {|v| / J_P^{v_\perp,*}(q), } & {\dim(G\cdot p)>0} \end{array}\right.
\end{equation}
where $v_\perp := V(q) + \{u\in T_q\mS_{p,r_0}: u\perp v\}$ is an $n$-plane in $T_qM$, $P : \C_{r_0}(G\cdot p) \to G\cdot p$ is the bundle projection, and $J_P^{v_\perp,*}(q)$ is the Jacobian of the submersion $P$ on $v_\perp$ at $q$. 
We can also define $\varphi$ on any other slice $\mS_{g\cdot p,r_0}$ in the same way. 

Note $G\cdot q$ intersects transversally with the slice $\mS_{p,r_0}$ for all $q\in \mS_{p,r_0}$. 
Thus, Lemma \ref{Lem: distribution on slice}(ii) suggests that $(dP)_q\llcorner V(q)$ has full rank, so $J_P^{v_\perp,*}(q)>0$ and $\varphi$ is well defined when $\dim(G\cdot p)>0$. 
Moreover, since $V(p)=T_pG\cdot p$ and $(dP)_p\llcorner T_pG\cdot p = id$, we have 
\begin{equation}\label{Eq: integrand is 1 at center}
	J_P^{v_\perp,*}(p) = 1\quad {\rm and} \quad \varphi(p,v)=  |v|, ~\forall 0\neq v\in T_pS_{p,r_0}. 
\end{equation}
Because $G$ acts by isometries and $V$ is $G$-invariant (Lemma \ref{Lem: distribution on slice}(i)), one verifies that 
\begin{equation}\label{Eq: integrand is G invariant}
	(dg(v))_\perp  = dg(v_\perp)\quad{\rm and}\quad  \varphi(g\cdot q, dg(v)) = \varphi(q,v)
\end{equation}
for all $g\in G$ by the chain rule. 

\begin{lemma}\label{Lem: integrand on slice}
	For $r_0>0$ small enough, $\varphi$ is a smooth $G_p$-invariant elliptic integrand on $\mS_{p,r_0}$ so that 
	\[\frac{\mH^n(\Gamma)}{\mH^{\dim(G\cdot p)}(G\cdot p)} =  {\underline{\underline \varphi}} (\Gamma\cap \mS_{p,r_0}):=\int_{\Gamma\cap \mS_{p,r_0}}\varphi(q, v(q))d\mH^{n-\dim(G\cdot p)}(q),\]
	for all $G$-invariant hypersurface $\Gamma$ in $\C_{r_0}(G\cdot p)$, where $v(q) $ is a unit normal of $\Gamma\cap \mS_{p,r_0}$ at $q$. 
	In particular, $\Gamma$ is a minimal $G$-hypersurface in $\C_{r_0}(G\cdot p)$ if and only if $\Gamma=G\cdot \Gamma_p$ and $\Gamma_p$ is a $G_p$-hypersurface in $\mS_{p,r_0}$ with $G_p$-stationarity for $\varphi$-integrand. 
\end{lemma}

\begin{proof}
	The lemma is trivially true if $\dim(G\cdot p)=0$, so we assume $\dim(G\cdot p)>0$. 
	It is clear that $\varphi$ is smooth and $G_p$-invariant on $\mS_{p,r_0}$. 
	By the smoothness and (\ref{Eq: integrand is 1 at center}), we see $\varphi$ is elliptic when $r_0>0$ is small enough. 
	Suppose $\Gamma$ is a $G$-invariant hypersurface in $\C_{r_0}(G\cdot p)$. 
	Then by Lemma \ref{Lem: distribution on slice}(ii) and the transversality between orbits and slices, we see $\Gamma_p := \Gamma\cap \mS_{p,r_0} = \Gamma\cap P^{-1}(p)$ is a $G_p$-invariant hypersurface in $ \mS_{p,r_0} $ and $T_q\Gamma=v(q)_\perp$, where $v(q) $ is a unit normal of $\Gamma_p\subset \mS_{p, r_0}$ at $q$. 
	Since $G$ acts by isometries, it is easy to verify that ${\underline{\underline \varphi}} (\Gamma_p) = {\underline{\underline \varphi}} (\Gamma_{g\cdot p})$ for all $g\in G$, and thus
	\begin{eqnarray*}
		\mH^n(\Gamma) & = & \int_{G\cdot p}\int_{\Gamma\cap P^{-1}(p_0)} 
		\frac{1}{J_P^{T_q\Gamma, *}(q)}     ~d\mH^{n-\dim(G\cdot p)}(q) ~d\mH^{\dim(G\cdot p)}(p_0)
		\\
		&=& \mH^{\dim(G\cdot p)}(G\cdot p)\cdot \int_{\Gamma_p}\varphi(q, v(q))d\mH^{n-\dim(G\cdot p)}(q)
	\end{eqnarray*}
	by the co-area formula. 
	In particular, a $G$-equivariant variation for the area functional in $\C_{r_0}(G\cdot p)$ is equivalent to a $G_p$-equivariant variation for the ${\underline{\underline \varphi}}$-functional in $\mS_{p,r_0}$. 
	Hence, the last statement follows from Lemma \ref{Lem: G-stationary}. 
\end{proof}

Now, we can use the arguments in \cite[Appendix]{white1987curvature} to construct $G_p$-invariant $\varphi$-stationary foliations on $\mS_{p,r}$ for $r>0$ small enough, which gives $G$-invariant minimal foliations on $\C_{r}(G\cdot p)$ by Lemma \ref{Lem: integrand on slice}. 
For the sake of completeness, we state the main result of this appendix and include the details, where we add a subscript $G, G_p$ to $C^{2,\alpha} , C^{2,\alpha}_0$ indicating that the functions are $G$ or $G_p$ invariant. 

\begin{proposition}\label{Prop: foliation}
	Let $(M^{n+1},g_{_M})$ be a closed Riemannian manifold and $G$ be a compact Lie group acting by isometries on $M$ such that $3\leq {\rm codim}(G\cdot p) \leq 7$, $\forall p\in M$. 
	Suppose $\Sigma\subset M$ is a closed smooth embedded minimal hypersurface with a $G$-invariant unit normal $\nu$ on $\C_{r_0}(G\cdot p)\cap \Sigma$, where $G\cdot p\subset \Sigma$ and $\C_{r_0}(G\cdot p)$ is defined as above. 
	Then there exists $\epsilon \in (0,r_0)$ so that if $0<r<\epsilon$ and 
	\[ w\in C^{2,\alpha}_G(\overline{\B_{r}^\Sigma(G\cdot p)} ), \quad \|w\|_{2,\alpha} < \epsilon, \]
	then for each $t\in [-1,1]$, there is a $C^{2,\alpha}_G$ function $v_t:\overline{\B_{r}^\Sigma(G\cdot p)}\to \mathbb{R}$ satisfying
	\begin{itemize}
		\item[(i)] ${\rm Graph}(v_t) := \exp_{\Sigma}^\perp(v_t\nu)$ is a $G$-invariant $C^{2,\alpha}$ minimal hypersurface;
		\item[(ii)] $v_t \llcorner \partial \B_{r}^\Sigma(G\cdot p) =  w + t$;
		\item[(iii)] $v_t$ depends on $t$ in a $C^1$ way so that $\{{\rm Graph}(v_t)\}_{t\in [-1,1]}$ foliates $\C_{r}(G\cdot p)$; 
		\item[(iv)] if $v^1_0, v^2_0$ correspond to $w_1,w_2$ respectively, then \[\sup_{q\in \overline{\B_{r}^\Sigma(G\cdot p)}}|v^1_0 - v^2_0| ~\leq~ 2 \cdot \sup_{q\in \partial \B_{r}^\Sigma(G\cdot p)} |w_1 - w_2|. \]
	\end{itemize}
\end{proposition}
\begin{proof}
	Let $\mathbb{B}_1:= \mathbb{B}^{n-\dim(G\cdot p)}_1(0)$, $\C_1 := \mathbb{B}_1\times [-1,1]$, where $\mathbb{B}^{n-\dim(G\cdot p)}_1(0)\subset {\bf N}_p^\Sigma G\cdot p$ is an Euclidean unit ball with $G_p$ acts on it orthogonally. 
	$F: \C_1\to \mS_{p,r_0}^\Sigma\times [-r_0, r_0]$ is a $G_p$-equivariant diffeomorphism given by the compound of $\exp_{G\cdot p}^{\Sigma,\perp}\times id$ and the $r_0$-dilation. 
	Now we use $F$ and $E$ to pull back the computations from $\mS_{p,r_0}$ to $\C_1$:
	\[ \tilde{g}:= (E\circ F)^* g_{_M}, \quad \tilde{\varphi}(x,v):=\varphi \big(E\circ F(x), d(E\circ F)(v) \big), \]
	where $x=(y,t)\in \C_1, 0\neq v\in T_x\C_1$. 
	
	For any $0<r<r_0$, define $\mu_r: \C_1\to \C_r$, $\mu_r(y,t)=(ry,rt)$ as the $r$-dilation and
	\[  g^r:=\frac{1}{r^2}\cdot (\mu_r)^* \tilde{g}, \quad \Phi_r(x,v):=\tilde{\varphi}(\mu_r(x), v),~\forall x\in \C_1, 0\neq v\in \mathbb{R}^{n+1-\dim(G\cdot p)}\cong T_x\C_1.\]
	Let $\nabla^r$ be the connection corresponding to $g^r$. 
	Then, as $r\to 0$, we see $g^r$ tends to the Euclidean metric $g^0:=\tilde{g}(0)$, and $\Phi_r$ tends to the area integrand $\Phi_0 :=\tilde{\varphi}(0,\cdot )=\varphi (p, d(E\circ F)(\cdot)) =|\cdot |_{g^0}$. 
	We also mention that $g^r$ and $\Phi_r$ are $G_p$-invariant. 
	Moreover, by the area formula, if $\Gamma\subset \C_1$ is a hypersurface, then 
	\begin{equation}\label{Eq: integrad scaling}
		{\underline{\underline {\Phi_r}}}(\Gamma ) = \int_{\Gamma} \Phi_r\big(x, \tilde{\nu}(x) \big)d\mH^n_{g^r}(x) = \frac{1}{r^n}
		\int_{\mu_r(\Gamma)} \tilde{\varphi} \big(x, \frac{1}{r}d\mu_r(\tilde{\nu} ) \big)d\mH^n_{\tilde{g}}(x) = \frac{1}{r^n} {\underline{\underline {\tilde \varphi}}}(\mu_r(\Gamma)),
	\end{equation}
	where $\tilde{\nu}$ and $\frac{1}{r}d\mu_r(\tilde{\nu} )$ are the unit normal of $(\Gamma, g^r)$ and $(\mu_r(\Gamma), \tilde{g})$ respectively, and we also used $\tilde{\nu}=\frac{1}{r}d\mu_r(\tilde{\nu} )$ by identifying tangent spaces with $\mathbb{R}^{n+1-\dim(G\cdot p)}$. 
	Therefore, $\Gamma\subset \C_1$ is a $\Phi_r$-stationary $G_p$-hypersurface under $g^r$ if and only if $\mu_r(\Gamma)\subset \C_r$ is a $\tilde{\varphi}$-stationary $G_p$-hypersurface under $\tilde{g}$, if and only if $E\circ F\circ \mu_r(\Gamma)$ is a $\varphi$-stationary $G_p$-hypersurface under $g_{_M}$. 
	
	For any $r\in [0, r_0)$, $x\in \mathbb{B}_1$, $y\in [-1,1]$, $z\in \R^{n-\dim(G\cdot p)}={\rm span} \big(\{ \frac{\partial}{\partial x_i} \}_{i=1}^{n-\dim(G\cdot p)} \big)$, let 
	\begin{itemize}
		\item $T_r(x,y,z):={\rm span}\big( \{ \frac{\partial}{\partial x_i} + \langle z, \frac{\partial}{\partial x_i}\rangle_{g^r(x,y)} \cdot \frac{\partial}{\partial t}  \}_{i=1}^{n-\dim(G\cdot p)} \big)$ be a hyperplane in $T_{(x,y)}\C_1$;
		\item $\nu_r(x,y,z)$ be the upper unit normal of $T(x,y,z)$ in $T_{(x,y)}\C_1$ under $g^r$;
		\item $A_r(x,y,z):= \Phi_r((x,y), \nu_r(x,y,z))$. 
	\end{itemize}
	Since $g^r$ and $\Phi_r$ are $G_p$-invariant, we have $ A_r(g\cdot x, y, dg(z)) = A_r(x,y,z)$ for all $g\in G_p$. 
	Then for any $f\in C^{2,\alpha}(\overline{\mathbb{B}_1})$, ${\rm Graph}(f) = \{(x,f(x)):x\in\overline{\mathbb{B}_1} \}$ satisfies $T_{(x,f(x))}{\rm Graph}(f) = T_r(x,f(x),\nabla^r f(x) )$ and 
	\[ {\underline{\underline {\Phi_r}}} ({\rm Graph}(f)) = \int_{\overline{\mathbb{B}_1}} A_r(x,f(x), \nabla^r f(x) ) d\mH^{n-\dim(G\cdot p)}(x).  \]
	If $f_s=f+s\eta \in C^{2,\alpha}(\overline{\mathbb{B}_1})$ is a variation of $f=f_0$ with $\eta\llcorner\partial \mathbb{B}_1 = 0 $, then 
	\[ \frac{d}{ds}\Big|_{s=0} {\underline{\underline {\Phi_r}}} ({\rm Graph}(f_s)) = \int_{\overline{\mathbb{B}_1}} H_r(f)(x)\cdot \eta (x) ~d\mH^{n-\dim(G\cdot p)}(x),\]
	where $H_r(f) = -\Div D_3A_r(x, f(x), \nabla^r f(x) ) + D_2 A_r(x, f(x), \nabla^rf(x) )\in C^{0,\alpha}(\overline{\mathbb{B}_1})$. 
	By the $G_p$-invariance, we have $A_r(x, f\circ g(x), \nabla^r f\circ g(x) ) = A_r(\cdot, f(\cdot), \nabla^rf(\cdot) )\circ g(x) $ and 
	\[H_r(f\circ g)(x) = H_r(f)(g\cdot x),\]
	for all $g\in G_p$. 
	In particular, $H_r(f)\in C^{0,\alpha}_{G_p}(\overline{\mathbb{B}_1})$ provided $f\in C^{2,\alpha}_{G_p}(\overline{\mathbb{B}_1})$.

	Define \[ \Psi : [0, r_0)\times \R \times C^{2,\alpha}_{G_p}(\overline{\mathbb{B}_1}) \times C^{2,\alpha}_{G_p,0}(\overline{\mathbb{B}_1}) \to C^{0,\alpha}_{G_p}(\overline{\mathbb{B}_1}) \]
	by $\Psi(r,t,w,u) = H_r(t+w+u)$. 
	Then $\Psi$ is $C^1$ (c.f. \cite[Appendix]{white1987space}) and 
	\[D_4\Psi(0, t, 0, 0)[v] = \frac{d}{ds}\Big|_{s=0} \Psi(0, t, 0, sv) = \frac{d}{ds}\Big|_{s=0} \Div_{g^0}\left(\frac{s\nabla^0 v }{\sqrt{1+s^2|\nabla^0  v|}} \right) = \triangle_{g^0} v .\]
	By standard Schauder estimates, the Laplace operator $\triangle_{g^0}$ in the Euclidean space is a Banach isomorphism between $C^{2,\alpha}_{0}(\overline{\mathbb{B}_1}) $ and $C^{0,\alpha}(\overline{\mathbb{B}_1})$. 
	Note, if $f\in C^{0,\alpha}(\overline{\mathbb{B}_1})$ is $G_p$-invariant and $\triangle_{g^0} v=f$, $v\llcorner\partial \mathbb{B}_1 = 0 $, then $\triangle_{g^0} (v\circ g)=f\circ g =f$, $v\circ g \llcorner\partial \mathbb{B}_1 = 0 $ for all $g\in G_p$. 
	Since $\triangle_{g^0}$ is an isomorphism (or by the maximum principle), we have $ v=v\circ g$, so $v$ is also $G_p$-invariant. 
	Therefore, 
	\[D_4\Psi(0, t, 0, 0): C^{2,\alpha}_{G_p,0}(\overline{\mathbb{B}_1}) \to C^{0,\alpha}_{G_p}(\overline{\mathbb{B}_1}) \]
	is an isomorphism. 
	By the implicit function theorem, for any $t_0\in\R$, there exists $\epsilon' = \epsilon'(t_0) \in (0, r_0)$ so that if $t\in (t_0-\epsilon', t_0+\epsilon' )$, $r\in [0, \epsilon')$ and $w\in C^{2,\alpha}_{G_p}(\overline{\mathbb{B}_1})$ with $\|w\|_{2,\alpha}<\epsilon'$ then we can find $u=u(r,t,w) \in C^{2,\alpha}_{G_p,0}(\overline{\mathbb{B}_1}) $ satisfies
	\begin{itemize}
		\item[(1)] $\Psi(r,t,w,u)=0$, i.e. ${\rm Graph}(t+w+u)$ is a $G_p$-invariant $\Phi_r$-stationary hypersurface in $(\C_1, g^r)$ with boundary value $t+w$;
		\item[(2)] $u$ depends in a $C^1$ way on $t$. 
	\end{itemize}
	Since $\Psi$ 	depends in a $C^1$ way on $t_0$, the constant $\epsilon'>0$ can be taken uniformly for $t_0\in [-1,1]$. 
	Then the $C^1$-dependence of $u$ on $t$ implies $\{{\rm Graph}(t+w+u)\}_{t\in [-1,1]}$ foliates $\C_\frac{1}{2}$. 
	Additionally, by taking $\epsilon'>0$ even smaller, we have 
	\begin{equation}\label{Eq: solution max-principle}
		\sup_{x\in \overline{\mathbb{B}_1}}|(t+w_1+u_1) - (t+w_2+u_2)| ~\leq~ 2 \cdot \sup_{x\in \partial \mathbb{B}_1}|w_1 - w_2|, 
	\end{equation}
	where $u_1=u(r,t,w_1)$, $u_2=u(r,t,w_2)$, $t\in [-1,1]$. 
	Indeed, if $r=0$, this inequality is valid with $1$ in place of $2$ on the right-hand side by the maximum principle. 
	Thus, the above inequality follows from the continuity.  
	
	For such $\epsilon'>0$, set $\epsilon =\frac{\epsilon'}{2}$.
	Then for any $t\in [-1,1]$, $0<r<\epsilon$, $ w\in C^{2,\alpha}_G(\overline{\B_{r}^\Sigma(G\cdot p)} )$ with $\|w\|_{2,\alpha} < \epsilon $, we can define $w' (x)= w(E\circ F \circ \mu_r(x))$ for any $x\in \overline{\mathbb{B}_1}$ so that $w'\in C^{2,\alpha}_{G_p}(\overline{\mathbb{B}_1})$ and $\|w'\|_{2,\alpha}<\epsilon$. 
	Then the above constructions give $u'= u'(r,t,w')\in C^{2,\alpha}_{G_p,0}(\overline{\mathbb{B}_1})$ satisfying (1)(2)(\ref{Eq: solution max-principle}) and $\C_{\frac{1}{2}}\subset \cup_{t\in [-1,1]} {\rm Graph}(t+w'+u') $. 
	Now we define $u''=u''(r,t,w')\in \C^{2,\alpha}_{G_p}(\mS_{p,r}^\Sigma)$ by $u''(q)=u'\left( \left(E\circ F\circ \mu_r \right)^{-1}(q) \right)$ for all $q\in \mS_{p,r}^\Sigma$. 
	After setting $u(g\cdot q) = u''(q)$ for all $g\in G$ and $q\in \mS_{p,r}^\Sigma$, we obtain a well-defined $G$-invariant $C^{2,\alpha}$ function $u=u(r,t,w)$ on $\B_{r}^\Sigma(G\cdot p)$. 
	Finally, take $v_t := t+w+u(r,t,w)$, and the proof is finished by the properties of $u'$, (\ref{Eq: integrad scaling}), and Lemma \ref{Lem: integrand on slice}. 
\end{proof}

\noindent
{\bf Declarations:}
\\
\noindent
{\bf Conflict of interest.} The author declares that there are no conflicts of interest. 
\\
\noindent
{\bf Availability of data.} Not applicable (no data are attached to this paper).


\bibliographystyle{abbrv}

\providecommand{\bysame}{\leavevmode\hbox to3em{\hrulefill}\thinspace}
\providecommand{\MR}{\relax\ifhmode\unskip\space\fi MR }
\providecommand{\MRhref}[2]{%
  \href{http://www.ams.org/mathscinet-getitem?mr=#1}{#2}}
\providecommand{\href}[2]{#2}

\bibliography{reference}   

\end{document}